\documentclass{amsart}
\usepackage{amsthm, amssymb, amsfonts, amscd}
\usepackage{graphicx}
\usepackage{stmaryrd} 
\usepackage{tikz}
\usepackage{tikz-cd}
\usepackage{bbold}
\usetikzlibrary{arrows}
\usepackage{verbatim}
\usepackage[all,arc]{xy}
\usepackage{enumerate}
\usepackage{mathrsfs, mathabx}
\usepackage{xcolor}[2007/01/21]
\usepackage{hyperref}
\usepackage{soul}
\usepackage{comment}
\usepackage{esint} 
\usepackage[all]{xy}
\usepackage[margin=1.0in]{geometry}
\usepackage[
   open,
   openlevel=2,
   atend
 ]{bookmark}[2011/12/02]
\usepackage{graphics}

\bookmarksetup{color=blue}

\theoremstyle{definition} 
\newtheorem{thm}{Theorem}[section]
\newtheorem{cor}[thm]{Corollary}
\newtheorem{prop}[thm]{Proposition}
\newtheorem{lem}[thm]{Lemma}

\theoremstyle{definition}
\newtheorem{defn}[thm]{Definition}

\newtheorem{rem}[thm]{Remark}

\newcommand{\Tr}{\mathrm{Tr}}

\newcommand{\Aut}{\mathrm{Aut}}

\newcommand{\bb}{\mathbb}

\newcommand{\Z}{\bb{Z}}
\newcommand{\bQ}{\bb{Q}}
\newcommand{\bR}{\bb{R}}
\newcommand{\bP}{\bb{P}}

\newcommand{\bC}{\bb{C}}
\newcommand{\bF}{\bb{F}}
\newcommand{\bE}{\bb{E}}

\newcommand{\zp}{\Z_p}

\newcommand{\mr}{\mathrm}

\newcommand{\mc}{\mathcal}

\newcommand{\unl}{\underline}
\newcommand{\cP}{\mc{P}}

\newcommand{\Sur}{\mathrm{Sur}}

\newcommand{\CL}{\mathrm{CL}}
\newcommand{\KL}{\mathrm{KL}}

\newcommand{\M}{\mr{M}}

\newcommand{\ep}{\epsilon}

\newcommand{\al}{\alpha}

\newcommand{\dt}{\delta}

\newcommand{\cok}{\mathrm{cok}}

\newcommand{\Hom}{\mathrm{Hom}}

\newcommand{\ol}{\overline}

\newcommand{\lt}{\left}
\newcommand{\rt}{\right}

\numberwithin{equation}{section}
\allowdisplaybreaks

\begin{document}

\title[Determinantal row-sparse matrices]{Distribution of the cokernels of determinantal row-sparse matrices}
\date{}
\author{Jungin Lee and Myungjun Yu}
\address{J. Lee -- Department of Mathematics, Ajou University, Suwon 16499, Republic of Korea \newline
M. Yu -- Department of Mathematics, Yonsei University, Seoul 03722, Republic of Korea}
\email{jileemath@ajou.ac.kr, mjyu@yonsei.ac.kr}

\begin{abstract}
We study the distribution of the cokernels of random row-sparse integral matrices $A_n$ 
according to the determinantal measure from a structured matrix $B_n$ with a parameter $k_n \ge 3$. 
Under a mild assumption on the growth rate of $k_n$, we prove that the distribution of the $p$-Sylow subgroup of the cokernel of $A_n$ converges to that of Cohen--Lenstra for every prime $p$.
Our result extends the work of A. Mészáros which established convergence to the Cohen--Lenstra distribution when $p \ge 5$ and $k_n=3$ for all positive integers $n$.
\end{abstract}

\maketitle

\section{Introduction}\label{Sec1}

As a higher-dimensional generalization of trees, Kalai \cite{Kal83} introduced the notion of hypertrees ($\bQ$-acyclic complexes). A finite $r$-dimensional simplicial complex $C$ is called a ($r$-dimensional) \textit{hypertree} if it has complete $(r-1)$-skeleton and $H_r(C, \bQ)=H_{r-1}(C, \bQ)=0$. If $C$ is a $r$-dimensional hypertree on $n$ vertices, then it has exactly $\binom{n-1}{r}$ $r$-faces and the homology group $H_{r-1}(C)$ is finite. Let $\mathcal{T}_r(n)$ be the set of all $r$-dimensional hypertrees on the vertex set $[n] = \{ 1, 2, \ldots, n \}$. By \cite[Theorem 1]{Kal83}, we have
\begin{equation} \label{eq1a}
\sum_{C \in \mathcal{T}_r(n)} |H_{r-1}(C)|^2 = n^{\binom{n-2}{r}}.
\end{equation}
When $r=1$, the above equation recovers Cayley's formula which states that the number of spanning trees on $n$ labeled vertices is $n^{n-2}$. 

The homology group $H_{r-1}(C)$ for $C \in \mathcal{T}_r(n)$ can be expressed as the cokernel of a certain integral matrix. Let $I_{n, r}$ be an $\binom{n-1}{r} \times \binom{n}{r+1}$ matrix whose rows are indexed by $r$-element subsets of $[n-1]$ and columns are indexed by $(r+1)$-element subsets of $[n]$. If $S$ is an $r$-element subset of $[n-1]$ and $S' = \{ s_0 < s_1 < \cdots < s_r \}$ is an $(r+1)$-element subset of $[n]$, the $(S,S')$ entry of the matrix $I_{n, r}$ is defined as
$$
I_{n,r}(S, S') = \left\{\begin{matrix}
(-1)^j & \text{ if } S'=S \cup \{ s_j \} \\
0 & \text{ if }S \not\subset S'
\end{matrix}\right. .
$$
Let $I_{n,r}^T[C]$ be the $\binom{n-1}{r} \times \binom{n-1}{r}$ submatrix of $I_{n,r}^T$ whose rows are indexed by the $r$-faces of $C$. If we regard $I_{n,r}^T[C]$ as an integral matrix, then we have $H_{r-1}(C) \cong \cok(I_{n,r}^T[C])$ by \cite[Lemma 2]{Kal83}.

Now we concentrate on the case $r=2$. Let $C_n$ be a random element in $\mathcal{T}_2(n)$ with distribution
$$
\bP(C_n = C) = \frac{|H_{1}(C)|^2}{n^{\binom{n-2}{2}}} = \frac{|\cok(I_{n,2}^T[C])|^2}{n^{\binom{n-2}{2}}}.
$$
(It is a probability distribution by \eqref{eq1a}.) Kahle and Newman \cite{KN22} conjectured that the $p$-Sylow subgroup of $H_1(C_n)$ converges to the Cohen--Lenstra distribution, i.e.
\begin{equation} \label{eq1b}
\lim_{n \to \infty} \bP(H_1(C_n)_p \cong G) = \nu_{\CL, p}(G) := \frac{1}{|\Aut(G)|}\prod_{i=1}^\infty (1-p^{-i})
\end{equation}
for every finite abelian $p$-group $G$. (For an abelian group $G$, denote the $p$-Sylow subgroup of $G$ by $G_p$.) This conjecture was disproved for $p=2$ by Mészáros \cite{Mes24a}, but it remains open for $p>2$. Note that Kahle, Lutz, Newman and Parsons \cite[Conjecture 5]{KLNP20} gave a similar conjecture for a uniform random element of $\mathcal{T}_2(n)$.

As an analogue, Mészáros \cite{Mes23} constructed a matrix $B_n$ which has a similar structure to $I_{n,2}^T$, but is easier to work with. For our purposes, we present a more general version of $B_n$. 
For each positive integer $n$, let $e_1, e_2, \ldots, e_n$ be the standard basis of $\bR^n$ and $k_n \ge 3$ be a positive integer.
The matrix $B_n$ is defined as follows. The columns of $B_n$ are indexed by $[n] := \{ 1, 2, \ldots, n \}$ and the rows of $B_n$ are indexed by $[n]^{k_n}$.
The row corresponding to $(b_1, b_2, \ldots, b_{k_n}) \in [n]^{k_n}$ is given by $e_{b_1} + e_{b_2} + \cdots + e_{b_{k_n}}$. For an $n$-element subset $Y$ of $[n]^{k_n}$, let $B_n[Y]$ denote the $n \times n$ submatrix of $B_n$ which consists of $n$ rows of $B_n$ indexed by $Y$. 
Let $X_n$ be the random $n$-element subset of $[n]^{k_n}$ with distribution
$$
\bP(X_n = Y) = \frac{\det(B_n[Y])^2}{\det(B_n^T B_n)}
$$
(it is a probability measure by the Cauchy--Binet formula) and $A_n$ be the random $n \times n$ integral matrix defined by $A_n = B_n[X_n]$.

For a sequence of random finite abelian $p$-groups $(G_n)_{n=1}^{\infty}$, we say $G_n$ \textit{converges to CL} if the distribution of $G_n$ converges to the Cohen--Lenstra distribution $\nu_{\CL, p}$ as $n \to \infty$. 
When $k_n=3$ for all $n$, Mészáros \cite[Theorem 1.1]{Mes23} proved that for every prime $p \ge 5$ the $p$-Sylow subgroup of $\cok(A_n)$ converges to CL. In this paper, we generalize this result to the case where $p \nmid k_n$ for all sufficiently large $n$ and $k_n$ does not grow too rapidly. 

\begin{thm}
\label{thm: cokernel distribution theorem 1_intro}
(Theorem \ref{thm: cokernel distribution theorem 1}) Let $G$ be a finite abelian group and $\cP$ be a finite set of primes including those dividing $|G|$. 
Assume that a sequence $(k_n)_{n=1}^{\infty}$ satisfies the following:
\begin{enumerate}
    \item for every prime $p$ in $\cP$, $p \nmid k_n$ for all sufficiently large $n$;
    \item for every $\ep>0$, $k_n < n^{\ep}$ for all sufficiently large $n$;
    \item \label{p=2 thm1a} if $2 \in \cP$, then for every $\dt>0$, $\dt \log\log n < k_n$ for all sufficiently large $n$.
\end{enumerate}
Then
$$
\lim_{n\to \infty} \bP\left(\bigoplus_{p \in \cP} \cok(A_n)_p \cong G\right) = \frac{1}{|\Aut(G)|}\prod_{p\in \cP} \prod_{i=1}^\infty (1-p^{-i})
= \prod_{p \in \cP} \nu_{\CL, p}(G_p).
$$
\end{thm}

\begin{cor} \label{cor1b}
Let $p$ be a prime such that $p \nmid k_n$ for all sufficiently large $n$ and $G$ be a finite abelian $p$-group. Assume that a sequence $(k_n)_{n=1}^{\infty}$ satisfies the following:
\begin{enumerate}
    \item for every $\ep>0$, $k_n < n^{\ep}$ for all sufficiently large $n$;
    \item \label{p=2 cor1b} if $p=2$, then for every $\dt>0$, $\dt \log\log n < k_n$ for all sufficiently large $n$.
\end{enumerate}
If we regard $A_n$ as a random matrix defined over the ring of $p$-adic integers $\zp$, then 
$$
\lim_{n\to \infty} \bP\left(\cok(A_n) \cong G\right) = \nu_{\CL, p}(G).
$$
\end{cor}

In Theorem \ref{thm: cokernel distribution theorem 1_intro}, we assume that $k_n$ does not grow too slowly when $2 \in \cP$. In particular, $k_n$ should not be a constant when we consider the $2$-Sylow subgroup of $\cok(A_n)$. This assumption is necessary in the following respect. 

When $k_n=3$ for all $n$, Mészáros \cite[Remark 5.6]{Mes23} observed that $\bE(\# \Sur(\cok(A_n), \Z/2\Z)) \ge 1+4e^{-2} + o(1)$, and based on this, predicted that the $2$-Sylow subgroup of $\cok(A_n)$ does not converge to CL, just as the $2$-Sylow subgroup of $H_1(C_n)$ does not converge to CL.
In Section \ref{Sec_not CL}, we confirm this prediction by showing more generally that the distribution of $\dim_{\bF_2} \ker \ol{A_n}$ has heavier tail than the Cohen--Lenstra distribution when $k_n \ge 3$ is a fixed odd positive integer for all $n$. 
Here and below, we write $\ol{A_n}$ for the reduction of $A_n$ modulo $2$.

\begin{thm} \label{thm7a_copy}
(Theorem \ref{thm7a}) Let $k \ge 3$ be an odd integer, $k_n=k$ for all $n$ and $r$ be a positive integer. Then for all sufficiently large $n$, 
\begin{equation*}
\bP(\dim_{\bF_2} \ker \ol{A_n} \ge r) 
\ge \frac{1}{4r!} \left ( \frac{2(k-1)}{e^{k-1}} \right )^r.
\end{equation*}
In particular, $\cok(A_n)_2$ does not converge to CL (for $p=2$). 
\end{thm}

According to Wood \cite[Theorem 3.1]{Woo19}, the limiting distribution of random finitely generated abelian groups is uniquely determined by their (surjective) moments if the moments are not too large. Consequently, Theorem \ref{thm: cokernel distribution theorem 1_intro} (and Corollary \ref{cor1b}) follows from the next theorem. For two groups $G_1$ and $G_2$, denote the set of all surjective group homomorphisms from $G_1$ to $G_2$ by $\Sur(G_1, G_2)$.

\begin{thm}
\label{thm: moment theorem 1_intro}
(Theorem \ref{thm: moment theorem 1}) Let $G$ be a finite abelian group. Assume that a sequence $(k_n)_{n=1}^{\infty}$ satisfies the following:
\begin{enumerate}
    \item $\gcd(|G|, k_n)=1$ for all sufficiently large $n$;
    \item for every $\ep>0$, $k_n < n^{\ep}$ for all sufficiently large $n$;
    \item if $|G|$ is even, then for every $\dt>0$, $\dt \log\log n < k_n$ for all sufficiently large $n$.
\end{enumerate}
Then
\begin{equation}
\label{eq: moment 1}
\lim_{n\to \infty} \bE(\# \Sur(\cok(A_n), G)) = 1.     
\end{equation}
\end{thm}
Note that it is necessary to assume that $\lim_{n \to \infty} k_n  = \infty$ to get \eqref{eq: moment 1} when $G= \Z/2\Z$ (see Proposition \ref{prop: kn has to go to infinity}). It would be interesting to know whether Theorem \ref{thm: moment theorem 1_intro} still holds under a weaker version of assumption (3), namely, only requiring that $\lim_{n \to \infty}k_n = \infty$. See also Proposition \ref{prop: partial answer}. 
A large part of the proof of Theorem \ref{thm: moment theorem 1_intro}
closely follows the approach of Mészáros in \cite{Mes23}; however, a more careful analysis was required to keep track of the effect of
$k_n$, since $k_n$ depends on $n$ and may vary accordingly.

The paper is organized as follows. In Section \ref{Sec2}, we extend the results of \cite[Section 3-4]{Mes23} to a large class of sequences $(k_n)_{n=1}^{\infty}$. 
In particular, Proposition \ref{prop: moments sum splits} reduces Theorem \ref{thm: moment theorem 1_intro} to proving 
\begin{equation} \label{eq_main_Sec1}
\lim_{n \to \infty} \sum_{\unl{n} \in B(n, G)} E(\unl{n}) = 1
\end{equation}
and 
\begin{equation} \label{eq_error_Sec1}
\lim_{n \to \infty} \sum_{\unl{n} \in B(n, H)} E(\unl{n}) = 0.
\end{equation}
See \eqref{eq: nearly-unifrom ball} for the definition of $H$-nearly-uniform ball $B(n,H)$.

We prove \eqref{eq_main_Sec1} in Section \ref{Sec_main term} and \eqref{eq_error_Sec1} in Section \ref{Sec_B1(n,H)} and \ref{Sec_B2(n,H)}. More precisely, we divide the set $B(n, H)$ into two parts $B_1(n,H)$ and $B_2(n, H)$ and prove
$$
\lim_{n \to \infty} \sum_{\unl{n} \in B_1(n, H)} E(\unl{n}) = 0 \quad\text{and}\quad 
\lim_{n \to \infty} \sum_{\unl{n} \in B_2(n, H)} E(\unl{n}) = 0
$$
in Section \ref{Sec_B1(n,H)} and \ref{Sec_B2(n,H)}, respectively. 
Note that the set $B_1(n,H)$ is empty when $|G|$ is odd, and hence Section \ref{Sec_B1(n,H)} contains a new ingredient which was not presented in \cite{Mes23}. 
We finish the proofs of the main theorems in Section \ref{Sec_CL all primes} and prove Theorem \ref{thm7a_copy} in Section \ref{Sec_not CL}.

\section{Decomposition of the moments into sums of probabilities over nearly-uniform balls} \label{Sec2}

We begin by introducing some notation. Recall that $k_n \ge 3$ depends on $n$, however we often write $k=k_n$ when there is no danger of confusion. Let $\unl{n} = (n_a)_{a \in G}$ be a tuple in $\Z_{\ge 0}^{G}$ such that $\sum_{a \in G} n_a = n$ and let
$$
G_+(\unl{n}) = \{ a \in G : n_a > 0\}. 
$$
When $\unl{n}$ is clear from context, we often write
$$
G_+ = G_+(\unl{n}).
$$
For each $a \in G$, let $n(1)_a := n_{-a}$, and $n(\ell)_a := \sum_{b \in G} n_b n(\ell-1)_{a+b}$ for $2 \le \ell \le k_n-1$. Inductively, for every $a \in G$ and $1 \le \ell \le k_n-1$, we have
$$
n(\ell)_a = \sum_{\substack{b_1, \ldots, b_{\ell} \in G \\ b_1 + \cdots +b_{\ell} = -a}}  n_{b_1}n_{b_2}\cdots n_{b_{\ell}}. 
$$
When $k_n=3$, our definitions of $n(1)_a$ and $n(2)_a$ correspond to $n_{-a}$ and $m_a$ in \cite{Mes23}, respectively.

The matrix \textit{associated to} $\unl{n}$ is defined to be the $|G_+| \times |G_+|$ matrix $M = M_{\unl{n}}$ whose rows and columns are indexed by $G_+$ and the entries $M(a,b)$ ($a, b \in G_+$) are given by
$$
M(a,b) := \begin{cases}
(k_n-1)n_a n(k_n-2)_{2a} + n(k_n-1)_a &\text{if $a=b$}, \\
(k_n-1)\sqrt{n_an_b}n(k_n-2)_{a+b} &\text{if $a\neq b$}.
\end{cases}
$$

Let $G$ be a finite abelian group, $q = (q_1, q_2, \ldots, q_n) \in G^n$ and $a \in G$. Define
$$
S(\unl{n}) = \{ q \in G^n : \# \{ i \in [n] : q_i = a \} = n_a \text{ for every } a \in G \}
$$
and
$$
E(\unl{n}) = \sum_{q \in S(\unl{n})} \bP(A_nq = 0). 
$$
For $q \in S(\unl{n})$, let $M_q := M_{\unl{n}}$ and $G_+(q) := G_+(\unl{n})$. 

\begin{lem}
\label{lem: diagonal entry of M bound}
Let $\unl{n} = (n_a)_{a \in G}$. For $q \in S(\unl{n})$ and $a \in G_+(q)$, we have
$$
M_q(a,a) \le k_n n(k_n-1)_a. 
$$
\end{lem}

\begin{proof}
We have
\begin{align*}
n(k_n-1)_a & = \sum_{\substack{b_1, \ldots, b_{k_n-1} \in G \\ b_1 + \cdots +b_{k_n-1} = -a}}  n_{b_1}n_{b_2}\cdots n_{b_{k_n-1}} \\
& \ge \sum_{\substack{b_1, \ldots, b_{k_n-2} \in G \\ b_1 + \cdots +b_{k_n-2} + a = -a}}  n_{b_1}n_{b_2}\cdots n_{b_{k_n-2}}n_a \\ 
& = n(1)_a \sum_{\substack{b_1, \ldots, b_{k_n-2} \in G \\ b_1 + \cdots +b_{k_n-2} = -2a}}  n_{b_1}n_{b_2}\cdots n_{b_{k_n-2}} \\
& = n(1)_an(k_n-2)_{2a}
\end{align*}
so $M_q(a,a) = (k_n-1)n(1)_a n(k_n-2)_{2a} + n(k_n-1)_a \le k_n n(k_n-1)_a$.
\end{proof}

We recall that (see Section 2.2 of \cite{Mes23})
\begin{equation}
\label{eq: moment is sum of probability}
\bE(\#\Sur(\cok(A_n), G)) = \sum_{q} \bP(A_nq = 0),    
\end{equation}
where the sum is over all $q = (q_1, \ldots, q_n) \in G^n$ such that $q_1, \ldots, q_n$ generate $G$. 
In order to make use of \eqref{eq: moment is sum of probability}, we first find a formula for $\bP(A_nq = 0)$.

\begin{lem}
\label{lem: prob expression}
Let $\unl{n}=(n_a)_{a \in G}$ and $q \in S(\unl{n})$. Then
$$
\bP(A_nq = 0) = \frac{1}{k_n}n^{-(k_n-1)n} \det(M) \prod_{a \in G_+} (n(k_n-1)_a^{n_a -1}). 
$$
Moreover, the matrix $M$ is positive semi-definite. 
\end{lem}

\begin{proof}
We closely follow the proof of \cite[Lemma 3.1]{Mes23}. 
We write $k = k_n$ for convenience. Let
$$
I_q = \{(x_1, \ldots, x_k) \in [n]^k : q_{x_1} + \cdots + q_{x_k} = 0\}
$$
and $B_{n,q}$ be the submatrix of $B_n$ which consists of the rows with indices in $I_q$. By the Cauchy--Binet formula, we have
$$
\bP(A_n q =0) = \bP(X_n \subset I_q) = \sum_{\substack{K \subset I_q \\ |K| =n}} \frac{(\det(B_n[K]))^2}{\det(B_n^T B_n)} = \frac{\det(B_{n,q}^TB_{n,q})}{\det(B_n^T B_n)}.
$$
Following the proof of \cite[Lemma 3.1]{Mes23}, we deduce that
$$
B_{n,q}^TB_{n,q}(i,j) = \begin{cases}
k(k-1)n(k-2)_{2q_i} + kn(k-1)_{q_i} &\text{if $i=j$}, \\
k(k-1)n(k-2)_{q_i+q_j} &\text{if $i\neq j$}.
\end{cases}
$$
For $a \in G_+$, define the vector $w_a \in \bR^{[n]}$ so that
$$
w_a(i) = \begin{cases}
1/\sqrt{n_a} & \text{if $~q_i = a$,} \\
0 & \text{otherwise.}
\end{cases}
$$
Also, writing $\langle,\rangle$ for the standard inner product in $\bR^{[n]}$, we define the subspace
$$
W_a := \{v \in \bR^{[n]} : \langle v, w_a\rangle = 0, \text{ and } v(i) = 0 ~\text{for all $i$ with $q_i \neq a$}\}.
$$
Then $W_a$ is a subspace of $\bR^{[n]}$ which is invariant under $B_{n,q}^TB_{n,q}$, and the matrix $B_{n,q}^TB_{n,q}$ acts on $W_a$ as multiplication by $kn(k-1)_{a}$. 
Furthermore, it is straightforward to see that $kM$ is the matrix representation of the restriction of $B_{n,q}^TB_{n,q}$ to the subspace spanned by $(w_a)$, with respect to the basis $(w_a)$.
Following the proof of \cite[Lemma 3.1]{Mes23}, we see that
$$
\det(B_{n,q}^TB_{n,q}) = \det(kM)\prod_{a \in G_+} (kn(k-1)_a)^{n_a - 1} = k^n\det(M)\prod_{a \in G_+} (n(k-1)_a)^{n_a - 1}.
$$
Taking $q=0$ (so $n_0=n$ and $n_a=0$ for each $a \in G \backslash \{0 \}$) in the above equation, we have
\begin{equation}
\label{eq: det of BnTBn}
\det(B_n^T B_n) = \det(B_{n,0}^TB_{n,0}) = 
k^n \cdot kn^{k-1} \cdot (n^{k-1})^{n-1} = k^{n+1}n^{(k-1)n} 
\end{equation}
and hence the lemma is proved.
The fact that $M$ is positive semi-definite can be proved as in the proof of \cite[Lemma 3.1]{Mes23}. 
\end{proof}

By Lemma \ref{lem: prob expression} and the formula $|S(\unl{n})| = \frac{n!}{\prod_{a \in G} n_a!}$, we have
\begin{equation}
\label{eq: E formula 1}
E(\unl{n}) = \frac{n!}{\prod_{a \in G} n_a!}\frac{n^{-(k_n-1)n}\det(M)}{k_n}\prod_{a \in G_+(\unl{n})} n(k_n-1)_a^{n_a - 1}. 
\end{equation}

\begin{lem}
\label{lem: E equal to 0 case}
For a given $\unl{n}$, suppose that there exists $a \in G_+$ such that $n(k_n-1)_a = 0$. Then $E(\unl{n}) = 0$. 
\end{lem}

\begin{proof}
We have
$n(k_n-1)_a = \sum_{b \in G}n_bn(k_n-2)_{a+b}$ so $n(k_n-1)_a=0$ implies that $n(k_n-2)_{a+b}= 0$ for all $b \in G_+$. Then $M(a,b)=0$ for all $b \in G_+$ by the definition of $M$, so $\det(M) = 0$ and thus $E(\unl{n})=0$.
\end{proof}

Similarly to the proof of Mészáros \cite{Mes23}, the notion of Kullback--Leibler divergence of probability measures, which we define next, will play a key role in our proof. 

\begin{defn}
Let $\nu, \mu$ be probability measures on a finite set $S$. The \textit{Kullback--Leibler divergence} of $\nu$ and $\mu$ is defined by
$$
D_{\KL}(\nu||\mu) := \sum_{x \in S} \nu(x)\log\left(\frac{\nu(x)}{\mu(x)}\right),
$$
where we interpret the summand $\nu(x)\log\left(\frac{\nu(x)}{\mu(x)}\right)$ as $0$ when $\nu(x) = 0$ and $D_{\KL}(\nu||\mu)$ is defined to be $\infty$ when there is $x \in S$ such that $\nu(x) \neq 0$ and $\mu(x) = 0$. 
\end{defn}

Throughout the paper, the letters $\nu$ and $\mu$ will always denote probability measures $\nu_{\unl{n}}$ and $\mu_{\unl{n}}$ on $G$ with a given $\unl{n}$ as follows, unless stated otherwise. 

\begin{defn}
For a given $\unl{n} = (n_a)_{a \in G} \in \Z_{\ge 0}^{G}$ such that $\sum_{a \in G} n_a = n$, the probability measures $\nu_{\unl{n}}$ and $\mu_{\unl{n}}$ on $G$ are defined by
$$
\nu_{\unl{n}}(a) = \frac{n_a}{n} \quad\text{and}\quad \mu_{\unl{n}}(a) = \frac{n(k_n-1)_a}{n^{k_n-1}}.
$$
For every subgroup $H$ of $G$, a probability measure $\nu_H$ on $G$ is defined by
$$
\nu_H(a) = \begin{cases}
\frac{1}{|H|} & ~\text{if $a \in H$}, \\
0 & ~\text{if $a \not\in H$}.
\end{cases}
$$
\end{defn}
Suppose that $n(k_n-1)_a >0$ for all $a \in G_+ = G_+(\unl{n})$. Then we have
\begin{equation}
\label{eq: E formula 2}
E(\unl{n}) = \al(\unl{n})\frac{\det(M)}{k_n\prod_{a \in G_+}n(k_n-1)_a}\exp(-n D_{\KL}(\nu_{\unl{n}}||\mu_{\unl{n}}))
\end{equation}
where
$$
\al(\unl{n}) = \frac{n!}{\prod_{a \in G}n_a!}\exp\left(n\sum_{a \in G_+} \nu_{\unl{n}}(a)\log \nu_{\unl{n}}(a)\right) \le 1.
$$
(The inequality $\al(\unl{n}) \le 1$ follows from \cite[Lemma 2.2]{CS04}.)
Since $M$ is positive semi-definite and $\Tr M \le k_n n^{k_n-1}$, we have
$$
\det(M) \le (\Tr M)^{|G_+|} \le k_n^{|G|}n^{(k_n-1)|G|}
$$
and thus
\begin{equation}
\label{eq: E inequality}
E(\unl{n})  \le k_n^{|G|}n^{(k_n-1)|G|}\exp(-n D_{\KL}(\nu_{\unl{n}}||\mu_{\unl{n}})).    
\end{equation}

\begin{lem}
\label{lem: Mes lemma 10}
Assume that $\gcd(|G|,k_n)=1$. Let $\unl{n} = (n_a)_{a \in G} \in \Z_{\ge 0}^{G}$ such that $\sum_{a \in G} n_a = n$. Let $\nu = \nu_{\unl{n}}$ and $\mu = \mu_{\unl{n}}$.
Then there is a positive real number $C_n = O_G(k_n^4)$ (which does not depend on the choice of $\unl{n}$) and a subgroup $H$ of $G$ such that the following two conditions hold. 
\begin{enumerate}
\item 
$\left | \nu_H(a) - \nu(a) \right | \le C_n\sqrt{D_{\KL}(\nu || \mu)}$ for every $a \in G$.
\item 
$\nu(G \backslash H) \le C_n D_{\KL}(\nu || \mu)$.    
\end{enumerate}
\end{lem}

\begin{proof}
We follow the proof of \cite[Lemma 4.1]{Mes23} with some modifications. 
For simplicity, we write $k=k_n$ in the proof. 
By Pinsker's inequality (\cite[Lemma 2.3]{Mes23}), we have $\dt := \sum_{x \in G}| \nu(x)-\mu(x)| \le 2 \sqrt{D_{\KL}(\nu || \mu)}$. For a character $\rho \in \hat{G} = \Hom(G, \bC^*)$, the Fourier transforms of $\nu$ and $\mu$ are given by $\hat{\nu}(\rho) = \sum_{a \in G} \rho(a) \nu(a)$ and $\hat{\mu}(\rho) = \sum_{a \in G} \rho(a) \mu(a)$. Then we have
$$
\hat{\mu}(\rho) 
= \sum_{a \in G} \rho(a) \frac{n(k-1)_a}{n^{k-1}}
= \sum_{a \in G} \sum_{\substack{b_1, \ldots, b_{k-1} \in G \\ b_1+\cdots+b_{k-1}=-a}} \prod_{i=1}^{k-1} \rho(b_i)^{-1} \frac{n_{b_i}}{n}
= \prod_{i=1}^{k-1} \left ( \sum_{b_i \in G} \rho(b_i)^{-1} \frac{n_{b_i}}{n} \right )
= \ol{(\hat{\nu}(\rho))^{k-1}}
$$
and $\left\| \hat{\nu}(\rho)-\hat{\mu}(\rho) \right\| \le \dt$ so $z = \hat{\nu}(\rho)$ satisfies the condition $|z - \ol{z^{k-1}}| \le \dt$.

Define $f(z)=z-\ol{z^{k-1}}$. The roots of $f(z)$ are $0$, $1$ and $e^{2\pi i\ell/k}$ for $1 \le \ell \le k-1$. Let $u(t)=t^{k-1}-t$. Then $u(0)=u(1)=0$ and $u'(t)=(k-1)t^{k-2}-1$ so $| u'(t) | \ge 1/2$ if $0 \le t < \left ( \frac{1}{2(k-1)} \right )^{\frac{1}{k-2}}$ or $t > \left ( \frac{3}{2(k-1)} \right )^{\frac{1}{k-2}}$. By the inequality $|f(z)| \ge |u(|z|)|$, 
$$
|f(z)| < \dt_0 := \frac{1}{2} \min \left ( \left ( \frac{1}{2(k-1)} \right )^{\frac{1}{k-2}}, 1 - \left ( \frac{3}{2(k-1)} \right )^{\frac{1}{k-2}} \right )
$$
implies that $|z| < 2 \dt_0$ or $||z|-1| < 2 \dt_0$. Note that 
$$
\dt_0 \le \frac{1}{4} \left ( \left ( \frac{1}{2(k-1)} \right )^{\frac{1}{k-2}} + 1 - \left ( \frac{3}{2(k-1)} \right )^{\frac{1}{k-2}} \right ) < \frac{1}{4}.
$$
We also have $\dt_0 > \frac{1}{3k}$ since $\dt_0 = \frac{1}{8}$ for $k=3$ and
$$
\dt_0 = \frac{1}{2} \left ( 1 - \left (\frac{3}{2(k-1)} \right )^{\frac{1}{k-2}} \right ) 
= \frac{1}{2} \frac{1 - \frac{3}{2(k-1)}}{\sum_{i=0}^{k-3} \left (\frac{3}{2(k-1)} \right )^{\frac{i}{k-2}}}
> \frac{1 - \frac{3}{2(k-1)}}{2(k-2)}
\ge \frac{1}{2k}
$$
for every $k \ge 4$.

Now we prove that if $|f(z)| < \dt_0$, then $|z-z_0| \le 4|f(z)|$ for some root $z_0$ of $f(z)$. If $|z| < 2 \dt_0 < \frac{1}{2}$, then $|f(z)| \ge |z|-|z|^{k-1} \ge \frac{|z-0|}{2}$. Now assume that $||z|-1| < 2 \dt_0 < \frac{1}{2}$. Let $z=re^{i \theta}$ (so $|z|=r$) and $\frac{k \theta}{\pi} = 2q+\ep$ for some $q \in \Z$ and $|\ep| \le 1$. Then
$$
|f(z)| = |re^{i \theta} - r^{k-1} e^{-i(k-1) \theta}| = r|e^{ik \theta} - r^{k-2}| = r|e^{i \pi \ep} - r^{k-2}| \ge r|r-1|.
$$
Since $r > 1 - 2 \dt_0 > \frac{1}{2}$ and $|f(z)| < \dt_0 < \frac{1}{4}$, we have $|\sin \pi \ep| \le |e^{i \pi \ep} - r^{k-2}| = \frac{|f(z)|}{r} < \frac{1}{2}$ so $|\ep| < \frac{1}{6}$. Thus
$$
|f(z)| = r|e^{i \pi \ep} - r^{k-2}| \ge r|\sin \pi \ep| \ge 3r|\ep| \ge |\ep|.
$$
For the root $z_0 = e^{\frac{2 \pi i q}{k}}$ of $f(z)$, we have 
$$
|z-z_0| \le |re^{i \theta}-e^{i \theta}| + |e^{i \theta}-e^{\frac{2 \pi i q}{k}}| 
\le |r-1| + \frac{\pi |\ep|}{k} 
\le \frac{|f(z)|}{r} + \frac{\pi}{k} |f(z)| 
\le 4|f(z)|.
$$

Assume that $G \neq \{ 1 \}$ and let $m \ge 2$ be the smallest positive integer such that $mG=0$ (so $\gcd(m,k)=1$). For $a \in G$ and $\rho \in \hat{G}$, $\rho(a) = e^{2 \pi i t/m}$ for some $t \in \Z$ so 
$$
\Re(\rho(a)e^{-2\pi i\ell/k}) 
= \cos\left( \frac{2 \pi(kt - m \ell)}{mk} \right)
\le \cos\left( \frac{2 \pi}{mk} \right)
$$ 
for each $1 \le \ell \le k-1$. This implies that for every $1 \le \ell \le k-1$ and $\rho \in \hat{G}$, we have
\begin{align*}
|\hat{\nu}(\rho) - e^{2\pi i\ell/k}|
& = \left | \sum_{a \in G} (\rho(a)e^{-2\pi i\ell/k}) \nu(a) - 1 \right | \\
& \ge 1 - \sum_{a \in G} \nu(a) \Re(\rho(a)e^{-2\pi i\ell/k}) \\
& \ge 1 - \cos\left( \frac{2 \pi}{mk} \right) \\
& > \frac{4 \pi^2}{3(mk)^2} =: N_k.    
\end{align*}
(The last inequality follows from the fact that $1 - \cos x > \frac{x^2}{3}$ for $|x| \le \frac{2\pi}{3}$ with $x\neq 0$.) 

Let $\dt_1 = \frac{1}{m^2|G|k^2}$ and $C_n = \frac{2}{\dt_1^2}$ (since $k = k_n$ depends on $n$, so does $C_n$). If $\sqrt{D_{\KL}(\nu || \mu)} \ge \frac{\dt_1}{\sqrt{2}}$, then $C_nD_{\KL}(\nu || \mu) \ge 1$ so the lemma is trivial. From now on, we assume that $\sqrt{D_{\KL}(\nu || \mu)} < \frac{\dt_1}{\sqrt{2}}$, which implies that 
$$
\dt \le \sqrt{2 D_{\KL}(\nu || \mu)} < \dt_1 < \dt_0.
$$
Then $z=\hat{\nu}(\rho)$ satisfies $|f(z)| \le \dt < \dt_0$ so $|z-z_0| \le 4 |f(z)| \le 4 \dt$ for some root $z_0$ of $f$. However, $|\hat{\nu}(\rho) - e^{2\pi i\ell/k}| > N_k > 4 \dt_1 > 4 \dt$ for each $1 \le \ell \le k-1$ so $z_0$ should be $0$ or $1$. 
Thus for every character $\rho \in \hat{G}$, we have
$$
|\hat{\nu}(\rho)| \le 4\dt \quad\text{or}\quad |\hat{\nu}(\rho)-1| \le 4\dt.
$$

Let $\hat{G}_1$ be the set of characters $\rho \in \hat{G}$ such that $|\hat{\nu}(\rho) - 1| \le 4 \dt$. For every $\rho \in \hat{G}_1$ and $a \in G \backslash \ker \rho$, we have $\Re (\rho(a)) \le \cos \frac{2 \pi}{m} \le 1 - \frac{8}{m^2}$ 
so $\Re(\hat{\nu}(\rho)) \le 1 - \frac{8}{m^2} \nu(G \backslash \ker \rho)$.
Therefore
$$
1 - 4 \dt \le \Re(\hat{\nu}(\rho))
\le 1 - \frac{8}{m^2} \nu(G \backslash \ker \rho)
$$
so $\nu(G \backslash \ker \rho) \le \frac{m^2 \dt}{2}$. 
For a subgroup $H = \bigcap_{\rho \in \hat{G}_1} \ker \rho$ of $G$, we deduce that 
$$
\nu(G \backslash H) \le \sum_{\rho \in \hat{G}_1} \nu(G \backslash \ker \rho) \le \frac{m^2 |G| \dt}{2}.
$$

Now we prove the two statements of the lemma. 
First, we claim that $\left | \hat{\nu}(\rho) - \hat{\nu_H}(\rho)\right | \le 4 \dt$ for all $\rho \in \hat{G}$.
If $\rho \in \hat{G}_1$, then $\hat{\nu_H}(\rho) = 1$ and
$$
|\hat{\nu}(\rho) - \hat{\nu_H}(\rho)| \le 4\dt.
$$
If $\rho \in \hat{G}\backslash \hat{G}_1$, $|\hat{\nu}(\rho)| \le 4 \delta$ by the definition of $\hat{G}_1$. If $H \subseteq \ker \rho$, then
$$
|\hat{\nu}(\rho)| \ge \nu(H) - \nu(G\backslash H) = 1-2\nu(G\backslash H) \ge 1 - m^2 |G| \dt 
> 1 - m^2 |G| \dt_1 = 1 - \frac{1}{k^2} > 4 \dt,
$$
which is a contradiction. Thus, we may choose $h \in H \backslash \ker \rho$. Then
$$
\hat{\nu_H}(\rho) = \frac{1}{|H|}\sum_{x \in H} \rho(x) = \frac{1}{|H|}\sum_{x \in H} \rho(h+x) = \rho(h)\hat{\nu_H}(\rho).
$$
Therefore, $\hat{\nu_H}(\rho) = 0$ and we again have
$$
|\hat{\nu}(\rho) - \hat{\nu_H}(\rho)| \le 4\dt. 
$$
Hence, the above claim is verified. 
This implies that
$$
\left | \nu(a) - \nu_H(a) \right |
= \left| \frac{1}{|G|}\sum_{\rho \in \hat{G}} \ol{\rho(a)}(\hat{\nu}(\rho) - \hat{\nu_H}(\rho)) \right|
\le 4 \dt
\le 4 \sqrt{2D_{\KL}(\nu || \mu)}
\le C_n \sqrt{D_{\KL}(\nu || \mu)}
$$
for every $a \in G$ so the first statement is true. 

Now we prove the second assertion. 
Note that
$$
\mu(a) = \sum_{\substack{b_1, \ldots, b_{k-1} \in G \\ b_1 + \cdots +b_{k-1} = -a}} \nu(b_1)\cdots \nu(b_{k-1}).
$$
Let $p = \nu(G \backslash H)$ and $q = \mu(G \backslash H)$. Letting $B_i = \{ (b_1, \ldots, b_{k-1} ) \in G^{k-1}: b_i \in G\backslash H~\text{and}~b_j \in H~\text{for all $j\neq i$}\}$, it is straightforward to see that
$$
q \ge \sum_{i=1}^{k-1}\sum_{(b_1, \ldots, b_{k-1} ) \in B_i}  \nu(b_1)\cdots \nu(b_{k-1}) = (k-1)p(1-p)^{k-2}.
$$
By the inequality
$$
p \le \frac{m^2 |G| \dt}{2} < \frac{m^2 |G| \dt_1}{2} = \frac{1}{2k^2} < 2 \dt_0 \le 1 - \lt ( \frac{3}{2(k-1)} \rt )^{\frac{1}{k-2}},
$$
it follows that
$$
q \ge (k-1)p(1-p)^{k-2} \ge (k-1)p \cdot \frac{3}{2(k-1)} = \frac{3p}{2}.
$$ 
By \cite[Lemma 2.2]{Mes23}, we have
$$
D_{\KL}(\nu || \mu) \ge f(p,q) := p \log \frac{p}{q} + (1-p) \log \frac{1-p}{1-q}.
$$
The function $f(p,q)$ is monotone increasing on $[p,1]$ (as a function of $q$) so $f(p,q) \ge f(p, \frac{3p}{2}) =: m(p)$. We have $m(0)=0$, $m'(0) = 0.5 - \log 1.5 > C_n^{-1}$ and $m''(p) = \frac{1}{(2-3p)^2(1-p)} \ge 0$ for every $0 \le p \le 1/2$ so $m(p) > C_n^{-1}p$. Finally, we conclude that $p = \nu(G\backslash H) \le C_n m(p) \le C_n D_{\KL}(\nu || \mu)$. 
\end{proof}

For a positive integer $n$, let
$$
D_n = \left\{ \unl{n}= (n_a)_{a \in G} \in \Z_{\ge 0}^{G} : \sum_{a \in G} n_a=n \text{ and the elements $a \in G$ with $n_a >0$ generate } G \right\}
$$
and
$$
D_n' = \{\unl{n} \in D_n : n(k_n-1)_a > 0 ~\text{for all}~ a \in G_+(\unl{n})\}.
$$
By \eqref{eq: moment is sum of probability} and Lemma \ref{lem: E equal to 0 case}, it follows that
\begin{equation}
\label{eq: moment Dn' sum}
\bE(\#\Sur(\cok(A_n), G)) = \sum_{\unl{n} \in D_n} E(\unl{n}) =  \sum_{\unl{n} \in D_n'} E(\unl{n}).        
\end{equation}

For $C_n = 2m^4|G|^2k_n^4 = O(k_n^4)$ given in the proof of Lemma \ref{lem: Mes lemma 10}, let
$$
t_n = (k_n-1)C_n \sqrt{|G|n\log n} \quad \text{and} \quad r_n= (k_n-1)^2C_n|G|\log n.
$$

Define
\begin{equation}
\label{eq: nearly-unifrom ball}
B(n,H) := \left\{ \unl{n} \in D_n' : \\| \nu_{\unl{n}}(a) - \nu_H(a) | \le \frac{t_n}{n} \text{ for every $a \in G$ } ~\text{and}~  \sum_{a \notin H} \nu_{\unl{n}}(a) \le \frac{r_n}{n} \right\},    
\end{equation}
which we call \emph{$H$-nearly-uniform ball}. 
\begin{lem}
\label{lem: disjoint balls}
Suppose that $k_n = O(n^{1/10 - \ep})$ for some $\ep>0$. Let $H_1$ and $H_2$ be distinct subgroups of $G$. Then for all sufficiently large $n$,
$$
B(n,H_1) \cap B(n, H_2) = \emptyset. 
$$
\end{lem}

\begin{proof}
We may assume that $|H_1| \ge |H_2|$, so there exists $g \in H_1 \backslash H_2$. Let $\unl{n} = (n_a)_{a \in G} \in D_n'$ and suppose that $\unl{n} \in B(n, H_1) \cap B(n,H_2)$. Then we have
$$
\left|n_g- \frac{n}{|H_1|}\right| \le t_n
\quad\text{and}\quad
|n_g| \le r_n.
$$
Now we have
$$
\frac{n}{|H_1|} \le t_n+r_n = O(k_n^5 \sqrt{n \log n} ) = O(n^{1 - 5 \ep} \sqrt{\log n}),
$$
which is true only for finitely many $n$. 
\end{proof}

\begin{prop}
\label{prop: moments sum splits}
Suppose that $\gcd(|G|, k_n) = 1$ for all sufficiently large $n$. Suppose that $k_n = O(n^{1/10 - \ep})$ for some $\ep >0$. Then we have
$$
\lim_{n \to \infty} \left ( \bE(\#\Sur(\cok(A_n), G))- \sum_{H \in \text{Sub}(G)} \sum_{\unl{n} \in B(n, H)} E(\unl{n}) \right ) = 0,
$$
where $\text{Sub}(G)$ denotes the set of all subgroups of $G$.
\end{prop}

\begin{proof}
Assume that $n$ is sufficiently large so that $\gcd(|G|, k_n) = 1$ and the sets $B(n,H)$ ($H \in \text{Sub}(G)$) are pairwise disjoint (by Lemma \ref{lem: disjoint balls}).
By Lemma \ref{lem: Mes lemma 10}, we see that if $D_{\KL}(\nu_{\unl{n}}||\mu_{\unl{n}}) \le \frac{(k_n-1)^2|G|\log n}{n}$ for some $\unl{n} \in D_n'$, then $\unl{n} \in B(n,H)$ for some subgroup $H$ of $G$. Define
$$
D_n'' = D_n' \backslash \bigcup_{H \in \text{Sub}(G)} B(n,H).
$$
By \eqref{eq: moment Dn' sum}, we have
$$
\sum_{\unl{n} \in D_n''} E(\unl{n}) = \bE(\#\Sur(\cok(A_n), G))- \sum_{H \in \text{Sub}(G)} \sum_{\unl{n} \in B(n, H)} E(\unl{n}).
$$
If $\unl{n} \in D_n''$, then \eqref{eq: E inequality} yields
$$
E(\unl{n}) \le k_n^{|G|} n^{(k_n-1)|G|}\exp(-(k_n-1)^2|G|\log n) = k_n^{|G|}n^{-(k_n-1)(k_n-2)|G|}. 
$$
By the inequality $|D_n''| \le |D_n| \le (n+1)^{|G|}$, it follows that
$$
0 \le \sum_{\unl{n} \in D_n''} E(\unl{n}) \leq \left ( \frac{k_n(n+1)}{n^{(k_n-1)(k_n-2)}} \right )^{|G|}
$$
and the right-hand side converges to $0$ as $n \to \infty$.
\end{proof}

For later use, we record here the size of $B(n,H)$. 

\begin{lem}
\label{lem: size of ball for H}    
For every subgroup $H$ of $G$, 
$$
|B(n,H)| = O\left(k_n^{6|G|} \sqrt{n}^{|H|-1}(\log n)^{|G|}\right).
$$
\end{lem}

\begin{proof}
Let $\unl{n} = (n_a)_{a\in G} \in B(n,H)$. Then we have $|n_a-\frac{n}{|H|}| \leq t_n$ for each $a \in H\backslash \{0\}$, $n_a \leq r_n$ for each $a \in G \backslash H$ and $n_0 = n - \sum_{a \in G \backslash \{ 0 \} } n_a$. 
These imply that
$$
|B(n,H)| = O\left(t_n^{|H|-1}r_n^{|G|-|H|}\right). 
$$
Now the assertion follows from $C_n = O(k_n^4)$ and the definitions of $r_n$ and $t_n$. 
\end{proof}

By Proposition \ref{prop: moments sum splits}, to show Theorem \ref{thm: moment theorem 1_intro}, it is enough to prove that
$$
\lim_{n \to \infty} \sum_{\unl{n} \in B(n, G)} E(\unl{n}) = 1
$$
and 
$$
\lim_{n \to \infty} \sum_{\unl{n} \in B(n, H)} E(\unl{n}) = 0
$$
for every proper subgroup $H$ of $G$. 
Section \ref{Sec_main term} to \ref{Sec_B2(n,H)} will be devoted to the proof of these equalities.

\section{Computing the moments: sum over \texorpdfstring{$B(n,G)$}{B(n,G)}} \label{Sec_main term}
In this section, we prove that if $k_n$ does not grow too rapidly, then 
\begin{equation*}
\lim_{n \to \infty} \sum_{\unl{n} \in B(n,G)} E(\unl{n}) = 1. 
\end{equation*}
More precisely, we prove the following. 
\begin{prop}
\label{prop: main moment}
Suppose that
\begin{equation}
\label{eq: assumption for computing main moment}
k_n = O\left(n^{\frac{1}{30} - \ep}\right) \text{ for some $\ep>0$}.
\end{equation}
Then
\begin{equation}
\label{eq: main moment}
\lim_{n \to \infty} \sum_{\unl{n} \in B(n,G)} E(\unl{n}) = 1. 
\end{equation}
\end{prop}
We assume \eqref{eq: assumption for computing main moment} throughout this section. 
Let us briefly explain Mészáros' idea (Section 5.2 of \cite{Mes23}) to prove \eqref{eq: main moment}. Let 
$G$ be an arbitrary finite abelian group. Let $\mathfrak{E}$ be the square matrix of size $|G|-1$ with all its entries given by $1$, and let $\mathfrak{I}$ be the identity matrix of size $|G|-1$. Define 
$$
Q = |G|(\mathfrak{E} + \mathfrak{I}).
$$ 
For $\unl{n} \in B(n,G)$, let $P(\unl{n})$ denote the projection of $\unl{n}$ to the $(|G|-1)$-tuple indexed by $G\backslash \{0\}$. Mészáros found an expression for $E(\unl{n})$ as follows: 
\begin{equation}
\label{eq: expression for summand of main moment}
E(\unl{n}) = (1+o(1))\frac{\sqrt{|G|}^{|G|}}{\sqrt{2\pi n}^{|G|-1}} \exp\left(-\frac{1}{2}y^TQy\right), 
\end{equation}
where $y = \frac{P(\unl{n})- \frac{n}{|G|} \cdot \mathbb{1}}{\sqrt{n}}$.\footnote{There is a typo in the definition of $y$ in \cite{Mes23}, where $\mathbb{1}$ should be replaced by $n\cdot \mathbb{1}$; a similar correction applies to the definition of $K_n$.}
Define
$$
K_n = \left\{ \frac{P(\unl{n})- \frac{n}{|G|} \cdot \mathbb{1}}{\sqrt{n}}: \unl{n} \in B(n,G)\right\}.
$$
Then, we have
$$
\sum_{\unl{n} \in B(n,G)} E(\unl{n}) = (1+o(1))\sum_{y \in K_n} \frac{\sqrt{|G|}^{|G|}}{\sqrt{2\pi n}^{|G|-1}} \exp\left(-\frac{1}{2}y^TQy\right). 
$$
Furthermore, define $f_n : \mathbb{R}^{G \backslash \{0\}} \to \mathbb{R}$ as
$$
f_n(x) = \begin{cases}
\frac{\sqrt{|G|}^{|G|}}{\sqrt{2\pi}^{|G|-1}}\exp\left(-\frac{1}{2}y^TQy\right) \quad &\text{if $x \in y+[0, \frac{1}{\sqrt{n}})^{G\backslash \{0\}}$ for some $y \in K_n$}, \\
0 &\text{otherwise}. 
\end{cases}
$$
Mészáros observed that 
\begin{equation}
\label{eq: pointwise convergence}
f_n(x) \longrightarrow f_\infty(x):=\frac{\sqrt{|G|}^{|G|}}{\sqrt{2\pi }^{|G|-1}} \exp\left(-\frac{1}{2}x^TQx\right) ~\text{ for all $x \in \mathbb{R}^{G \backslash \{0\}}$ (pointwise convergence)}. 
\end{equation}
Then by applying the dominated convergence theorem and the Gaussian integral formula, Mészáros finally proved \eqref{eq: main moment} when $k_n = 3$ for all positive integers $n$ (see Section 5.2 of \cite{Mes23} for details). In our case, $k_n$ changes as $n$ varies, so we have to check whether \eqref{eq: expression for summand of main moment} and \eqref{eq: pointwise convergence} still hold in our setting, and this is what we will do in the rest of this section. As before, we closely follow Mészáros' argument.

Given $(\nu(a))_{a \in G\backslash \{0\}}$, let
$$
\nu(0) = 1 - \sum_{a \in G\backslash \{0\}} \nu(a) \quad\text{and}\quad \mu(a) = \sum_{\substack{b_1, \ldots, b_{k_n-1} \in G \\ b_1 + \cdots + b_{k_n-1} = -a}} \nu(b_1)\nu(b_2)\cdots\nu(b_{k_n-1}). 
$$
Let
$$\mathbb{R}':= \left\{(h_a)_{a \in G \backslash \{0\} }: 0\le h_a < 1~\text{and}~ \sum_{a \in G \backslash \{0\}} h_a \le 1 \right\}.$$
Define a function 
$$
f: \mathbb{R}' \to \mathbb{R}
$$ 
by sending $(\nu(a))_{a \in G\backslash\{0\}}$ to $D_\KL(\nu||\mu)$.

\begin{lem} (A generalization of \cite[Lemma 5.5]{Mes23})
The following statements hold. 
\begin{enumerate}
\item
We have $f\left(\frac{\mathbb{1}}{|G|} \right) = 0$.
\item 
The gradient of $f$ at $\frac{\mathbb{1}}{|G|}$ is $0$. 
\item 
The Hessian matrix of $f$ at $\frac{\mathbb{1}}{|G|}$ is $Q$. 
\item 
$Q$ is positive definite and $\det(Q) = |G|^{|G|}$.
\end{enumerate}

\end{lem}

\begin{proof}
For convenience, we write $k= k_n$.
Recall that 
$$
D_\KL(\nu||\mu) = \sum_{x \in G} \nu(x)\log\left(\frac{\nu(x)}{\mu(x)}\right).
$$
If $\nu(x) = 1/|G|$ for all $x \in G\backslash \{0\}$, we have
$$
\nu(0) = \frac{1}{|G|} \quad\text{and}\quad \mu(a) = \left(\frac{1}{|G|}\right)^{k-1}|G|^{k-2} = \frac{1}{|G|} \text{ for all $a \in G$}. 
$$
Therefore, $f(\bb{1}/|G|) = 0$. 
For $a \in G\backslash\{0\}$, let $\partial_a$ denote the partial derivative with respect to $\nu(a)$. For every $x \in G$, the product rule for derivatives implies that 
$$
\partial_a(\mu(x)) = (k-1)\left( \sum_{\substack{c_1, \ldots, c_{k-2} \in G \\  c_1+\cdots + c_{k-2} = -a-x}}\nu(c_1)\cdots \nu(c_{k-2}) - \sum_{\substack{c_1, \ldots, c_{k-2} \in G \\  c_1+\cdots + c_{k-2} = -x}}\nu(c_1)\cdots \nu(c_{k-2})\right).
$$
To ease the notation, we will abbreviate (as there is no danger of confusion)
$$
\sum_{x}\nu(c_1)\cdots \nu(c_{i}) = \sum_{\substack{c_1, \ldots, c_{i} \in G \\  c_1+\cdots + c_{i} = x}}\nu(c_1)\cdots \nu(c_{i})
$$
for every $i \ge 0$. Then we see that
\begin{align*}
\partial_a f & = \log \nu(a) - \log \nu(0) - \log \mu(a) + \log \mu(0) \\
& - \sum_{x\in G}\frac{\nu(x)}{\mu(x)}(k-1)\left( \sum_{-a-x}\nu(c_1)\cdots \nu(c_{k-2}) -  \sum_{-x}\nu(c_1)\cdots \nu(c_{k-2})\right).
\end{align*}
Then it follows that the gradient of $f$ at $\bb{1}/|G|$ is $0$. 
For every $a \in G \backslash \{ 0 \}$, we have
{\footnotesize
\begin{align*}
\partial_a\partial_a f = & \frac{1}{\nu(a)} + \frac{1}{\nu(0)} - \frac{(k-1)(\underset{-2a}{\sum} \nu(c_1)\cdots \nu(c_{k-2}) -  \underset{-a}{\sum} \nu(c_1)\cdots \nu(c_{k-2}))}{\mu(a)} \\ 
& + \frac{(k-1)(\underset{-a}{\sum}\nu(c_1)\cdots \nu(c_{k-2}) -  \underset{0}{\sum}\nu(c_1)\cdots \nu(c_{k-2}))}{\mu(0)}      \\
& -(k-1)\sum_{x \in G}\frac{\nu(x)(k-2) (\sum_{-2a-x}\nu(d_1)\cdots \nu(d_{k-3}) -  \sum_{-a-x}\nu(d_1)\cdots \nu(d_{k-3})     )}{\mu(x)}\\
& + (k-1)\sum_{x \in G}\frac{\nu(x)(k-2) (\sum_{-a-x}\nu(d_1)\cdots \nu(d_{k-3}) -  \sum_{-x}\nu(d_1)\cdots \nu(d_{k-3})     )}{\mu(x)} \\
& - (k-1)\left(\frac{\underset{-2a}{\sum} \nu(c_1)\cdots \nu(c_{k-2}) -  \underset{-a}{\sum} \nu(c_1)\cdots \nu(c_{k-2})}{\mu(a)} - \frac{\underset{-a}{\sum}\nu(c_1)\cdots \nu(c_{k-2}) -  \underset{0}{\sum}\nu(c_1)\cdots \nu(c_{k-2})}{\mu(0)}  \right) \\
& + \sum_{x \in G}\frac{\nu(x)}{\mu(x)^2}(k-1)^2\left(\sum_{-a-x}\nu(c_1)\cdots\mu(c_{k-2}) - \sum_{-x}\nu(c_1)\cdots\mu(c_{k-2}) \right)^2. 
\end{align*}
}
In particular, $\partial_a \partial_a f|_{\bb{1}/|G|} =2 |G|$. For every $a \neq b \in G\backslash\{0\}$, we have
{\footnotesize
\begin{align*}
\partial_b\partial_a f = & \frac{1}{\nu(0)} - (k-1)\frac{\underset{-b-a}{\sum}\nu(c_1)\cdots \nu(c_{k-2}) - \underset{-a}{\sum}\nu(c_1)\cdots \nu(c_{k-2})}{\mu(a)} \\
& + (k-1)\frac{\underset{-b}{\sum}\nu(c_1)\cdots \nu(c_{k-2}) - \underset{0}{\sum}\nu(c_1)\cdots \nu(c_{k-2})}{\mu (0)} \\
& - (k-1)\sum_{x \in G} \frac{\nu(x)(k-2)\left(\underset{-a-b-x}{\sum}\nu(d_1)\cdots \nu(d_{k-3}) - \underset{-a-x}{\sum}\nu(d_1)\cdots \nu(d_{k-3})\right)}{\mu(x)} \\
& + (k-1)\sum_{x \in G} \frac{\nu(x)(k-2)\left(\underset{-b-x}{\sum}\nu(d_1)\cdots \nu(d_{k-3}) - \underset{-x}{\sum}\nu(d_1)\cdots \nu(d_{k-3})\right)}{\mu(x)} \\
& - (k-1)\left(\frac{\underset{-b-a}{\sum}\nu(c_1)\cdots \nu(c_{k-2}) - \underset{-b}{\sum}\nu(c_1)\cdots \nu(c_{k-2})}{\mu(b)} - \frac{\underset{-a}{\sum}\nu(c_1)\cdots \nu(c_{k-2}) - \underset{0}{\sum}\nu(c_1)\cdots \nu(c_{k-2})}{\mu(0)}  \right) \\
& + \sum_{x \in G} \frac{\nu(x)(k-1)^2}{\mu(x)^2}\left(\underset{-a-x}{\sum}\nu(c_1)\cdots \nu(c_{k-2}) - \underset{-x}{\sum}\nu(c_1)\cdots \nu(c_{k-2}) \right)\left(\underset{-b-x}{\sum}\nu(c_1)\cdots \nu(c_{k-2}) - \underset{-x}{\sum}\nu(c_1)\cdots \nu(c_{k-2}) \right).
\end{align*}
}
In particular, $\partial_b \partial_a f|_{\bb{1}/|G|} =|G|$. Finally, it is straightforward to see that (4) holds.
\end{proof}

If $x \in \mathbb{R}'$ with $|x_a| \le t_n/n$ for all $a \in G\backslash \{0\}$, then by Taylor's expansion, it follows that
\begin{equation}
\label{eq: Taylor expansion}
nf\left(\frac{\mathbb{1}}{|G|} + x\right) = \frac{n}{2}x^TQx+ O\left(\frac{t_n^3}{n^2}\right). 
\end{equation}

As $t_n/\sqrt{n} = (k_n-1)C_n\sqrt{|G|\log n} \to \infty$ as $n \to \infty$, for every $x \in \mathbb{R}^{G \backslash \{0\}}$ there exists an integer $n_x$ such that $n > n_x$ implies the existence of $y \in K_n$ such that $0\le x_a - y_a < 1/\sqrt{n}$ for all $a \in G\backslash \{0\}$ (i.e., $x \in y + [0, \frac{1}{\sqrt{n}})^{G\backslash \{0\}}$). For such an $y$, we have
$$
|x^T Q x - y^TQy| \le \sum_{a,b \in G\backslash\{0\}}|Q(a,b)(x_ax_b - y_ay_b)| 
= O\left(\frac{t_n}{n}\right). 
$$
Since $k_n = O(n^{1/30 -\ep})$, we have $\lim_{n \to \infty}\frac{t_n}{n} = 0$ so 
$$\lim_{n \to \infty} (x^T Q x - y^TQy) = 0.$$ 
Therefore, we see that \eqref{eq: pointwise convergence} holds in our situation as well.

Let $M_{\mathrm{uni}}$ be the matrix $M_{\unl{n}}$ for $\unl{n} = \frac{n}{|G|}\bb{1}$. Then $M_{\mathrm{uni}}$ is clearly a square matrix of size $|G|$. Furthermore, the diagonal entries of $M_{\mathrm{uni}}$ are equal to $n^{k_n-1}(\frac{k_n-1}{|G|^2} + \frac{1}{|G|})$ and the off-diagonal entries are $\frac{(k_n-1)n^{k_n-1}}{|G|^2}$. It follows that
$$
\det(M_{\mathrm{uni}}) = \frac{k_n n^{(k_n-1)|G|}}{|G|^{|G|}}. 
$$
So, we have
$$
\frac{\det(M_{\mathrm{uni}})}{k_n\prod_{a \in G}n(k_n-1)_a} = 1. 
$$

Now let $\unl{n} \in B(n,G)$ and let $m= m_{\unl{n}}$ and $M= M_{\unl{n}}$. It follows from \eqref{eq: Taylor expansion} that
\begin{align*}
nD_{\KL}(\nu_{\unl{n}}||\mu_{\unl{n}}) = nf\left(\frac{P(\unl{n})}{n}\right) & = \frac{n}{2}\left(\frac{P(\unl{n})}{n} - \frac{\mathbb{1}}{|G|} \right)^TQ\left(\frac{P(\unl{n})}{n} - \frac{\mathbb{1}}{|G|} \right) + O\left(\frac{t_n^3}{n^2}\right)   \\
& = \frac{1}{2}y^TQy + O\left(\frac{t_n^3}{n^2}\right)
\end{align*}
where $y = \frac{P(\unl{n})- \frac{n}{|G|} \cdot \mathbb{1}}{\sqrt{n}}$.
Using the assumption that $k_n = O(n^{1/30- \ep})$, we get $nD_{\KL}(\nu_{\unl{n}}||\mu_{\unl{n}}) = \frac{1}{2}y^TQy + o(1)$, and it follows that
$$
\exp(-nD_{\KL}(\nu_{\unl{n}}||\mu_{\unl{n}})) = (1+o(1)) \exp\left(-\frac{1}{2}y^TQy \right). 
$$
Since $\unl{n} \in B(n,G)$, we have by Stirling's formula
$$
\al(\unl{n}) = (1+o(1)) \frac{\sqrt{|G|}^{|G|}}{\sqrt{2\pi n}^{|G|-1}}.
$$
For every $a,b \in G$, it follows from Lemma \ref{lem: lemma for section 2} below that 
$$
M(a,b) - M_{\mathrm{uni}}(a,b) = O(t_nn^{k_n-2}k_n^2).
$$
From this, it is straightforward to see that
$$
\det(M) = \left(k_n + O\left(\frac{t_n k_n^3}{n}\right)\right)\left(\frac{n^{k_n-1}}{|G|}\right)^{|G|} = k_n\left(\frac{n^{k_n-1}}{|G|}\right)^{|G|}(1+o(1)).
$$
Noting $n(k_n-1)_a = \frac{n^{k_n-1}}{|G|}+ O(t_nn^{k_n-2}k_n)$ (cf. Lemma \ref{lem: lem for 4.1(3)}), it follows that
$$
\prod_{a \in G} n(k_n-1)_a = \left(\frac{n^{k_n-1}}{|G|}\right)^{|G|}\left(1 + O\left(\frac{t_nk_n}{n}\right) \right) = \left(\frac{n^{k_n-1}}{|G|}\right)^{|G|}(1 + o(1)).
$$
Finally, \eqref{eq: expression for summand of main moment} follows from \eqref{eq: E formula 2}. 
As remarked earlier, the rest follows exactly as in Section 5.2 of \cite{Mes23}.

\begin{lem}
\label{lem: lemma for section 2}
For $\unl{n} \in B(n,G)$, let $M= M_{\unl{n}}$. Then for every $a,b \in G$, we have
$$
M(a,b) - M_{\mathrm{uni}}(a,b) = O(t_nn^{k_n-2}k_n^2). 
$$
\end{lem}

\begin{proof}
We only give a proof when $a=b$ and it can be proved similarly when $a\neq b$. By definition, we have
$$
|M(a,a) - M_{\mathrm{uni}}(a,a)| \le \frac{(k_n-1)(n+|G|t_n)^{k_n-1}}{|G|^2}+ \frac{(n+|G|t_n)^{k_n-1}}{|G|} - \frac{(k_n-1)n^{k_n-1}}{|G|^2} - \frac{n^{k_n-1}}{|G|}.
$$
Since $\lim_{n \to \infty} \frac{t_nk_n}{n} = 0$, the lemma follows from Lemma \ref{lem: lem for 4.1}.
\end{proof}



\section{Computing the moments: bounding the sum over \texorpdfstring{$B_1(n,H)$}{B1(n,H)}} \label{Sec_B1(n,H)}
Throughout this section, we assume the following.
\begin{enumerate}
\item
For every $\dt >0$, $\dt \log\log n< k_n$ for all sufficiently large $n$.
\item 
$k_n = O(n^{\frac{1}{24}- \ep})$ for some $\ep>0$.
\end{enumerate}
Note in particular that the second condition implies that
$$
\lim_{n \to \infty} \frac{k_nt_nr_n}{n} = 0.
$$

Let $H$ be a proper subgroup of $G$. Define
$$
B_1(n,H) = \{\unl{n} \in B(n,H): \text{there exists $g \in G\backslash H$ such that $2g \in H$ and $(g+H) \cap G_+(\unl{n}) \neq \emptyset$} \}
$$
and
$$
B_2(n,H) = B(n,H) \backslash B_1(n,H). 
$$
The goal of this section is to prove that
\begin{equation*}
\lim_{n \to \infty}\sum_{\unl{n} \in B_1(n, H)} E(\unl{n}) = 0
\end{equation*}
under the assumptions on the growth rate of $k_n$ as above.

\begin{rem}
\label{rem: remark about B1}
When $[G:H]$ is odd, it is clear that $B_1(n,H) = \emptyset$ by definition. However, it is straightforward to see that $B_1(n,H) \neq \emptyset$ if $[G:H]$ is even and $n$ is sufficiently large. Since we have no restriction on a finite abelian group $G$, we need to take $B_1(n,H)$ into account, whereas in \cite{Mes23}, this was unnecessary because $|G|$ was assumed to be odd there.
\end{rem}

\begin{lem}
\label{lem: DKL bound}
There exists $n_0>0$ such that for every $n > n_0$ and $\unl{n} \in B_1(n,H)$,
$$
D_{\KL}(\nu_{\unl{n}}||\mu_{\unl{n}}) \ge \frac{k_n}{n|G|}. 
$$
\end{lem}

\begin{proof}
Let $\unl{n} \in B_1(n,H)$. By the definition of $B_1(n,H)$, there exists $g \in G \backslash H$ such that $2g \in H$ and $(g+H) \cap G_+(\unl{n}) \neq \emptyset$. Let $a \in (g+H) \cap G_+(\unl{n})$. 
In particular, we have $a+H = -a+H$ and
$$
1 \le n_a \le r_n = (k_n-1)^2 C_n |G|\log n. 
$$
Let $p = \nu_{\unl{n}}(a)$ and $q = \mu_{\unl{n}}(a)$. By \cite[Lemma 2.2]{Mes23}, we have
$$
D_{\KL}(\nu_{\unl{n}}||\mu_{\unl{n}}) \ge p \log \frac{p}{q} + (1-p) \log \frac{1-p}{1-q}. 
$$
It follows from Lemma \ref{lem: lem for 4.1(2)} below that
$$
q = \frac{n(k_n-1)_a}{n^{k_n-1}}= \frac{(k_n-1)s}{|H|n} + O\left(\frac{t_nr_nk_n^2}{n^2}\right), 
$$
where
$$
s :=\sum_{b \in a+H} n_b = \sum_{b \in -a+H} n_b. 
$$
Note that we have
$$
1\le n_a \le s \le r_n. 
$$
It follows that for all sufficiently large $n$,
\begin{align*}
   D_{\KL}(\nu_{\unl{n}}||\mu_{\unl{n}}) & \ge \frac{n_a}{n}\log \frac{\frac{n_a}{n}}{\frac{(k_n-1)s}{|H|n} + O\left(\frac{t_nr_nk_n^2}{n^2}\right)} + \left(1-\frac{n_a}{n}\right)\left(\log\left(1-\frac{n_a}{n}\right) - \log \left(1- \frac{(k_n-1)s}{|H|n} - O\left(\frac{t_nr_nk_n^2}{n^2}\right)\right)\right) \\
   & \ge \frac{n_a}{n}\log \frac{\frac{n_a}{n}}{\frac{2(k_n-1)s}{|H|n}} + \left(1-\frac{n_a}{n}\right)\left(-\frac{n_a}{n}+\frac{(k_n-1)s}{|H|n} + O\left(\frac{t_nr_nk_n^2}{n^2}\right)\right)
\end{align*}
where we use the Taylor expansion $\log(1-x) = -x - x^2/2 - \cdots = -x + O(x^2)$ near $x=0$ for the last inequality. Then we see that for all sufficiently large $n$,
\begin{align*}
D_{\KL}(\nu_{\unl{n}}||\mu_{\unl{n}}) &\ge \frac{n_a}{n}\log \frac{1}{2k_nr_n} + \frac{(k_n-1)n_a}{|H|n} -\frac{n_a}{n} +  O\left(\frac{t_nr_nk_n^2}{n^2}\right)   \\
& = \frac{n_a}{n}\left(\frac{k_n-1}{|H|} - \log2k_ nr_n - 1\right) +  O\left(\frac{t_nr_nk_n^2}{n^2}\right)  \\
& \ge \frac{1}{n}\left(\frac{k_n-1}{|H|} - \log2k_nr_n - 1\right) +  O\left(\frac{t_nr_nk_n^2}{n^2}\right) \\ 
& \ge \frac{k_n}{n|G|}.
\end{align*}
Note that the last two inequalities hold by the fact that $r_n=O(k_n^6 \log n)$ and the assumptions on $k_n$ at the beginning of this section. This completes the proof. 
\end{proof}

\begin{lem}
\label{lem: lem for 4.1}  
Let $\{a_n\}$ and $\{b_n\}$ be sequences of positive integers and $c$ be a real number.  
Suppose that 
$$
\lim_{n \to \infty}\frac{a_n b_n}{n} = 0. 
$$
Then
$$
\left(n + ca_n \right)^{b_n} - n^{b_n} = O\left( n^{b_n-1}a_nb_n \right). 
$$
\end{lem}

\begin{proof}
Assume that $n$ is large enough so that $\left|\frac{ca_nb_n}{n}\right| < \frac{1}{2}$. Then we have
\begin{equation*}
\left|\left(n + ca_n \right)^{b_n} - n^{b_n}\right|
\le \sum_{i=1}^{b_n} n^{b_n - i} (|c|a_n)^i b_n^i
\le n^{b_n - 1} |c|a_n b_n \frac{1}{1-\frac{|c|a_n b_n}{n}}
\le 2  n^{b_n - 1} |c|a_n b_n. \qedhere
\end{equation*}
\end{proof}

\begin{lem}
\label{lem: lem for 4.1(2)}
Let $\unl{n} \in B(n,H)$, 
$a \in G \backslash H$ and $$m: = \sum_{b \in -a + H} n_{b}.$$ Let
\begin{align*}
B & = \{(b_1, \ldots, b_{k_n-1}) \in G^{k_n-1} : b_1+ \cdots + b_{k_n-1} = -a \text{ and $k_n-2$ of $b_i$'s are in $H$}\}, \\   
B^c & = \{(b_1, \ldots, b_{k_n-1}) \in G^{k_n-1} : b_1+ \cdots + b_{k_n-1} = -a \text{ and at most $k_n-3$ of $b_i$'s are in $H$}\}.  
\end{align*}
Then we have the following.
\begin{enumerate}
\item
$$\sum_{(b_1, \ldots, b_{k_n-1}) \in B} n_{b_1}\cdots n_{b_{k_n-1}} = (k_n-1)\frac{n^{k_n-2}}{|H|}m + O(n^{k_n-3}t_nr_nk_n^2).$$
\item 
$$\sum_{(b_1, \ldots, b_{k_n-1}) \in B^c} n_{b_1}\cdots n_{b_{k_n-1}} = O(n^{k_n-3}r_n^2 k_n^2).$$
\item 
$$
n(k_n-1)_a = (k_n-1)\frac{n^{k_n-2}}{|H|}m + O(n^{k_n-3}t_nr_nk_n^2).
$$
\end{enumerate}
\end{lem}

\begin{proof}
Assume that $n$ is large enough so that $n > t_n|H|$.
Since $\unl{n} \in B(n, H)$, we have $\left |n_b - \frac{n}{|H|} \right | \le t_n$ for every $b \in H$ and $\left | n_b \right | \le r_n$ for every $b \notin H$. We also have $m  \le r_n$. For simplicity, we write $(b) := (b_1, \ldots, b_{k_n-1})$. 
For $1\le i \le k_n-1$, let
$$
B_i := \{(b) \in B:  b_i \in -a+H~\text{and}~ b_j \in H ~\text{for all $j \neq i$}\}. 
$$
For $c \in -a+H$, define
$$
B_1(c) := \{(b) \in \mathcal{B}_1 : b_1 = c \}. 
$$
Then we have
$$
|B_1(c)|= |H|^{k_n-3}
$$
since for each $(b_1, \ldots, b_{k_n-1}) \in B_i(c)$, the first coordinate is fixed as $b_1=c$, we are free to choose $b_2, \ldots, b_{k_n-2} \in H$ and then $b_{k_n-1} = -c -\sum_{j=2}^{k_n-2} b_{j} - a \in H$ is determined. 
Thus we have
$$
n_c\left(\frac{n}{|H|} -t_n \right)^{k_n-2}|H|^{k_n-3}
\le \sum_{(b) \in B_1(c)} n_{b_1}\cdots n_{b_{k_n-1}}
\le n_c\left(\frac{n}{|H|} +t_n \right)^{k_n-2}|H|^{k_n-3},
$$
so it follows that
$$
m\left(\frac{n}{|H|} -t_n \right)^{k_n-2}|H|^{k_n-3}
\le \sum_{(b) \in B_1} n_{b_1}\cdots n_{b_{k_n-1}}
\le m\left(\frac{n}{|H|} +t_n \right)^{k_n-2}|H|^{k_n-3}.
$$
(By symmetry, the same holds for all $B_i$.) Since $B$ is a disjoint union of $B_1, \ldots, B_{k_n-1}$, we have
$$
(k_n-1)\left(\frac{n}{|H|} -t_n \right)^{k_n-2}|H|^{k_n-3}m
\le \sum_{(b) \in B} n_{b_1}\cdots n_{b_{k_n-1}}
\le (k_n-1)\left(\frac{n}{|H|} +t_n \right)^{k_n-2}|H|^{k_n-3}m.
$$
By Lemma \ref{lem: lem for 4.1}, we have
\begin{align*}
\left| \sum_{(b) \in B} n_{b_1}\cdots n_{b_{k_n-1}} - \frac{(k_n-1)n^{k_n-2}m}{|H|} \right| 
& \le \frac{(k_n-1)m}{|H|}\left((n + |H|t_n)^{k_n-2} - n^{k_n-2} \right) \\
& = O\left(n^{k_n-3} t_n r_nk_n^2\right),
\end{align*}
so (1) is true. For (2), note that
\begin{align*}
0 \le \sum_{(b) \in B^c} n_{b_1}\cdots n_{b_{k_n-1}}
& \le
\sum_{j=3}^{k_n} \left(\frac{n}{|H|} + t_n \right)^{k_n-j} r_n^{j-1} \binom{k_n-1}{j-1}|H|^{k_n-j} |G|^{j-2} \\
& \le\sum_{j=3}^{k_n} \left(\frac{n}{|H|} + t_n \right)^{k_n-j} r_n^{j-1} k_n^{j-1} |H|^{k_n-j} |G|^{j-2} \\
& \le (n + |H|t_n)^{k_n-3} r_n^{2}|G|k_n^2 \frac{1}{1 - \frac{k_nr_n|G|}{n}} 
\end{align*}
and $\frac{k_nr_n|G|}{n} < \frac{1}{2}$ when $n$ is sufficiently large. It is easy to see that (2) follows from the above inequality and Lemma \ref{lem: lem for 4.1}. Finally, (3) is immediate from (1) and (2). 
\end{proof}

For later use, we also estimate $n(k_n-2)_a$ for $a \in H$ when $\unl{n} \in B(n,H)$. 

\begin{lem}
\label{lem: lem for 4.1(3)}    
Let $\unl{n} \in B(n,H)$ and $a \in H$. 
Then 
$$
n(k_n-2)_a = \frac{n^{k_n-2}}{|H|}  + O(n^{k_n-3}t_nk_n). 
$$
\end{lem}

\begin{proof}
We argue similarly as in the proof of Lemma \ref{lem: lem for 4.1(2)}. 
Let 
\begin{align*}
\mathfrak{B} & = \{(b_1, \ldots, b_{k_n-2}) \in G^{k_n-2} : b_1+ \cdots + b_{k_n-2} = -a \text{ and $b_i \in H$ for all $1\le i \le k_n-2$}\}, \\   
\mathfrak{B}^c & = \{(b_1, \ldots, b_{k_n-2}) \in G^{k_n-2} : b_1+ \cdots + b_{k_n-2} = -a \text{ and at most $k_n-4$ of $b_i$'s are in $H$}\}.  
\end{align*}    
Assume that $n$ is large enough so that $n > t_n|H|$.
Since $\unl{n} \in B(n, H)$, we have $\left |n_b - \frac{n}{|H|} \right | \le t_n$ for every $b \in H$ and $\left | n_b \right | \le r_n$ for every $b \notin H$ so
$$
\left(\frac{n}{|H|} -t_n \right)^{k_n-2}|H|^{k_n-3}
\le \sum_{(b_1, \ldots, b_{k_n-2}) \in \mathfrak{B}} n_{b_1}\cdots n_{b_{k_n-2}}
\le \left(\frac{n}{|H|} +t_n \right)^{k_n-2}|H|^{k_n-3}
$$
and
\begin{align*}
0 \le \sum_{(b_1, \ldots, b_{k_n-2}) \in \mathfrak{B}^c} n_{b_1}\cdots n_{b_{k_n-2}}
& \le
\sum_{j=4}^{k_n} \left(\frac{n}{|H|} + t_n \right)^{k_n-j} r_n^{j-2} \binom{k_n-2}{j-2}|H|^{k_n-j} |G|^{j-3} \\
& \le\sum_{j=4}^{k_n} \left(\frac{n}{|H|} + t_n \right)^{k_n-j} r_n^{j-2} k_n^{j-2} |H|^{k_n-j} |G|^{j-3} \\
& \le (n + |H|t_n)^{k_n-4} r_n^{2}|G|k_n^2 \frac{1}{1 - \frac{k_nr_n|G|}{n}}.
\end{align*}
Now one can proceed as in the proof of Lemma \ref{lem: lem for 4.1(2)} to derive that
$$
\sum_{(b_1, \ldots, b_{k_n-2}) \in \mathfrak{B}} n_{b_1}\cdots n_{b_{k_n-2}} = \frac{n^{k_n-2}}{|H|} + O(n^{k_n-3}t_nk_n)
$$
and
\begin{equation*}
\sum_{(b_1, \ldots, b_{k_n-2}) \in \mathfrak{B}^c} n_{b_1}\cdots n_{b_{k_n-2}}
 = O (n^{k_n-4}r_n^2k_n^2). \qedhere
\end{equation*}
\end{proof}

\begin{prop}
\label{prop: moment for B1}
Suppose that the following statements hold. 
\begin{enumerate}
\item
For every $\dt >0$, $\dt \log\log n< k_n$ for all sufficiently large $n$.
\item 
$k_n = O(n^{\frac{1}{24}- \ep})$ for some $\ep>0$.
\end{enumerate}
Then
$$
\lim_{n \to \infty}\sum_{\unl{n} \in B_1(n,H)} E(\unl{n}) = 0.
$$
\end{prop}

\begin{proof}
Let $\unl{n} \in B_1(n,H)$. Recall that
$$
E(\unl{n}) = \al(\unl{n})\frac{\det(M)}{k_n\prod_{a \in G_+}n(k_n-1)_a}\exp(-n D_{\KL}(\nu_{\unl{n}}||\mu_{\unl{n}})), 
$$
where $G_+ = G_+(\unl{n})$ and
$$
\al(\unl{n}) = \frac{n!}{\prod_{a \in G}n_a!}\exp\left(n\sum_{a \in G_+} \nu(a)\log \nu(a)\right).
$$
By Stirling's formula, we have as $\unl{n} \in B(n,H)$,
\begin{equation}
\label{eq: Stirling formula for alpha}
\al(\unl{n}) = O\left(\frac{\sqrt{n}}{\prod_{a \in G_+} \sqrt{n_a}}\right) = O\left(\frac{\sqrt{n}}{\left(n/|H|-t_n\right)^{\frac{|H|}{2}}}\right) = O\left( n^{\frac{1-|H|}{2}}\right).    
\end{equation}
Since $M$ is positive semi-definite, it follows from the Hadamard's inequality \cite[Theorem 7.8.1]{HJ13} and Lemma \ref{lem: diagonal entry of M bound} that
\begin{equation}
\label{eq: 3.6.1}
\frac{\det(M)}{\prod_{a \in G_+} n(k_n-1)_a} \le \frac{\prod_{a \in G_+} M(a,a)}{\prod_{a \in G_+} n(k_n-1)_a} \le k_n^{|G|}.
\end{equation}
Then by Lemma \ref{lem: DKL bound}, we have
$$
E(\unl{n}) = O\left(\frac{k_n^{|G|}}{\sqrt{n}^{|H|-1}e^{\frac{k_n}{|G|}}} \right). 
$$
By Lemma \ref{lem: size of ball for H}, it follows that
$$
\sum_{\unl{n} \in B_1(n,H)} E(\unl{n}) = O\left(\frac{k_n^{7|G|}(\log n)^{|G|}}{e^{\frac{k_n}{|G|}}}\right).
$$
Now the proposition follows from assumption (1).
\end{proof}

\section{Computing the moments: bounding the sum over \texorpdfstring{$B_2(n,H)$}{B2(n,H)}} \label{Sec_B2(n,H)}
Throughout this section, we assume that for every $\ep >0$, $k_n < n^{\ep}$ for all sufficiently large $n$. 
Recall that
$$
B_1(n,H) = \{\unl{n} \in B(n,H): \text{there exists $g \in G\backslash H$ such that $2g \in H$ and $(g+H) \cap G_+(\unl{n}) \neq \emptyset$} \}. 
$$
and
$$
B_2(n,H) = B(n,H) \backslash B_1(n,H). 
$$
The goal of this section is to prove that 
\begin{equation}
\label{eq: Sec5 main}
\lim_{n \to \infty}\sum_{\unl{n} \in B_2(n, H)} E(\unl{n}) = 0
\end{equation}
under the assumption on $k_n$ by adopting the idea of Section 5.1 in \cite{Mes23}.

Assume that $n$ is large enough so that $n > |H|t_n$ and let $\unl{n} \in B_2(n,H)$. Since $\unl{n} \in B_2(n, H) \subseteq B(n,H)$, we have $|n_a - \frac{n}{|H|}| \le t_n < \frac{n}{|H|}$, $n_a>0$ for every $a \in H$ so $H \subseteq G_+ =G_+(\unl{n})$. 
For $g \in G \backslash H$, if $(g+H)$ and $G_+$ intersect then $2g \not\in H$ so $g+ H \neq -g +H$. 
Thus we can find $g_1, g_2, \ldots, g_h \in G \backslash H$ such that $G_+$ intersect $F_i = (g_i + H) \cup (-g_i + H)$ for every $1\le i \le h$, but $G_+$ does not intersect any coset $g + H$ other than the following $2h+1$ distinct cosets
$$
H, g_1 + H, \ldots, g_h + H, -g_1 + H, \ldots, -g_h+ H.
$$
We write
$$
\ell:= \left|\{1\le i \le h: G_+ \cap (g_i + H) \neq \emptyset ~~\text{and}~~ G_+ \cap (-g_i + H) \neq \emptyset  \}\right|.
$$

\begin{lem} 
\label{lem: lemma 5.2 of Mes23}
Let $\unl{n} \in B_2(n,H)$, $\ell$ be as above and $M=M_{\unl{n}}$ be the matrix associated to $\unl{n}$. Then 
$$
\frac{\det(M)}{\prod_{a \in G_+} n(k_n-1)_a} = O\left(k_n^{|G|}r_n^{|G|-|H|}\left(\frac{t_n}{n}\right)^\ell\right). 
$$
\end{lem}

\begin{proof}
We closely follow the proof of \cite[Lemma 5.2]{Mes23}. 
Let $M_i$ be the submatrix of $M$ determined by the rows and columns indexed by $F_i \cap G_+$ (let $F_0 = H$). As in the proof of \cite[Lemma 5.2]{Mes23}, we have
\begin{equation}
\label{eq: 3.1.0}
\frac{\det(M)}{\prod_{a \in G_+} n(k_n-1)_a} \le \prod_{i=0}^h \frac{\det(M_i)}{\prod_{a \in F_i \cap G_+} n(k_n-1)_a}. 
\end{equation}
As $M_i$ is positive semi-definite (Lemma \ref{lem: prob expression}), it follows from Hadamard's inequality \cite[Theorem 7.8.1]{HJ13} and Lemma \ref{lem: diagonal entry of M bound} that
\begin{equation}
\label{eq: 3.1.1}
\frac{\det(M_i)}{\prod_{a \in F_i \cap G_+} n(k_n-1)_a} \le k_n^{|F_i \cap G_+|}.  \end{equation}

Now suppose that $G_+$ intersects both $g_i+H$ and $-g_i + H$. Let
$$
s_1 = \sum_{b \in g_i + H}n_b \quad\text{and}\quad s_2 = \sum_{b \in -g_i + H}n_b.
$$
By assumption, $1 \le s_1, s_2 \le r_n$. Let $a \in (g_i + H) \cap G_+$. Then by Lemma \ref{lem: lem for 4.1(2)}(3),
$$
n(k_n-1)_a = (k_n-1)\frac{n^{k_n-2}}{|H|}s_2 + O(n^{k_n-3}t_nr_nk_n^2). 
$$
By a discussion after \eqref{eq: Sec5 main}, we have $2a \not\in H$. 
By Lemma \ref{lem: lem for 4.1(2)}(3) (but replacing $k_n$ with $k_n-1$),
$$
(k_n-1)n_an(k_n-2)_{2a} = O(k_n r_n k_n n^{k_n-3}r_n) = O(n^{k_n-3}r_n^2 k_n^2).
$$
Therefore, it follows that
$$
M_i(a,a) = (k_n-1)n_an(k_n-2)_{2a} + n(k_n-1)_a = (k_n-1)\frac{n^{k_n-2}}{|H|}s_2 + O(n^{k_n-3}t_nr_n k_n^2).
$$
Similarly, if $a \in (-g_i + H) \cap G_+$, then
$$
M_i(a,a) = (k_n-1)\frac{n^{k_n-2}}{|H|}s_1 + O(n^{k_n-3}t_nr_n k_n^2).
$$
If $a \in (g_i+H) \cap G_+$ and $b \in (-g_i + H) \cap G_+$, then $a+b \in H$ so Lemma \ref{lem: lem for 4.1(3)} implies that
\begin{align*}
M_i(a,b) = M_i(b,a) & = (k_n-1) \sqrt{n_an_b}\left(\frac{n^{k_n-2}}{|H|} + O(n^{k_n-3}t_nk_n)\right)\\
& = (k_n-1) \sqrt{n_an_b}\frac{n^{k_n-2}}{|H|} + O(n^{k_n-3}t_nr_nk_n^2).
\end{align*}
If $a,b \in (g_i+H) \cap G_+$ and $a \neq b$, then $a+b \not\in H$ as $2g_i \not\in H$, so by Lemma \ref{lem: lem for 4.1(2)}(3) we have
$$
M_i(a,b) =  O(n^{k_n-3}r_n^2 k_n^2) = O(n^{k_n-3}t_nr_nk_n^2).
$$
Similarly, if  $a,b \in (-g_i+H) \cap G_+$ and $a \neq b$, then 
$$
M_i(a,b) =  O(n^{k_n-3}r_n^2 k_n^2) = O(n^{k_n-3}t_nr_nk_n^2).
$$
Define the vector $v \in \bR^{F_i \cap G_+}$ so that
$$
v(a) = \begin{cases}
 \sqrt{n_a}  & \text{if $~a \in (g_i+H) \cap G_+$,}\\
 -\sqrt{n_a} & \text{if $~a \in (-g_i+H) \cap G_+$.}
\end{cases}
$$
By the computation in \cite[Lemma 5.2]{Mes23}, we have
$$
\frac{v^TM_iv}{||v||_2^2} = O\left(n^{k_n-3}t_nr_nk_n^2\right). 
$$

As in the proof of \cite[Lemma 5.2]{Mes23}, the smallest eigenvalue of $M_i$ is at most $O(n^{k_n-3}t_nr_nk_n^2)$. Furthermore, all the other eigenvalues are at most $\Tr(M_i) = O(n^{k_n-2}k_nr_n)$. It follows that 
$$
\det(M_i) = O(n^{(k_n-2)|F_i \cap G_+| - 1}r_n^{|F_i \cap G_+|}t_nk_n^{|F_i \cap G_+|+1}).
$$
For $a \in G_+ \cap (g_i + H)$, we have
$$
n(k_n-1)_a = (k_n-1)\frac{n^{k_n-2}}{|H|}s_2 + O(n^{k_n-3}t_nr_nk_n^2) \ge \frac{n^{k_n-2}k_n}{2|H|}
$$
when $n$ is sufficiently large and the same inequality holds for $a \in G_+ \cap (-g_i + H)$. Hence, when $n$ is sufficiently large, we have
$$
\frac{\det(M_i)}{\prod_{a \in F_i \cap G_+} n(k_n-1)_a} \le\frac{ O(n^{(k_n-2)|F_i \cap G_+| - 1}r_n^{|F_i \cap G_+|}t_nk_n^{|F_i \cap G_+|+1})}{(\frac{n^{k_n-2}k_n}{2|H|})^{|F_i \cap G_+|}} = O\left(r_n^{|F_i\cap G_+|}\frac{t_nk_n}{n}\right).
$$
Combining this with \eqref{eq: 3.1.0} and \eqref{eq: 3.1.1}, we obtain that
\begin{equation*}
\frac{\det(M)}{\prod_{a \in G_+} n(k_n-1)_a} = O\left(k_n^{|G|-\ell}r_n^{|G|-|H|}\left(\frac{t_n k_n}{n}\right)^\ell\right) =  O\left(k_n^{|G|}r_n^{|G|-|H|}\left(\frac{t_n}{n}\right)^\ell\right). \qedhere
\end{equation*}
\end{proof}

\begin{lem}
\label{lem: lemma 5.3 of Mes23}
Let $\unl{n} \in B(n,H)$. Suppose that there exists $a \in G_+$ such that $(-a+H) \cap G_+ = \emptyset$. Then for all sufficiently large $n$, we have
$$
D_{\KL}(\nu_{\unl{n}}||\mu_{\unl{n}}) \ge \frac{\log n}{2n}.
$$    
\end{lem}

\begin{proof}
Since $H \subseteq G_+$, we have $a \not\in H$. 
The proof of \cite[Lemma 5.3]{Mes23} works with a minor change as follows.
Recall that
$$
n(k_n-1)_a = \sum_{\substack{b_1, \ldots, b_{k_n-1} \in G \\b_1 + \cdots + b_{k_n-1} = -a}} n_{b_1}\cdots n_{b_{k_n-1}}. 
$$
Assume that $(b_1, \ldots, b_{k_n-1}) \in G^{k_n-1}$ satisfies $b_1+\cdots +b_{k_n-1} = -a$. If $k_n-2$ of $b_i$'s are in $H$, then there exists $i \in [k_n-1]$ such that $b_i \in -a+H$ so $n_{b_i}=0$ (and so $n_{b_1}\cdots n_{b_{k_n-1}}=0$) by the assumption of the lemma. Thus we have
$$
n(k_n-1)_a = \sum_{\substack{b_1, \ldots, b_{k_n-1} \in G \\b_1 + \cdots + b_{k_n-1} = -a \\\text{At least $2$ of $b_i$'s are in $G \backslash H$}}} n_{b_1}\cdots n_{b_{k_n-1}} \le \sum_{j=2}^{k_n-1} \binom{k_n-1}{j} \left(\frac{n}{|H|} + t_n\right)^{k_n-1-j} r_n^j |H|^{k_n-1-j}|G|^{j-1},
$$
where the summand of the right-hand side bounds the sum of $n_{b_1}\cdots n_{b_{k_n-1}}$ for those $(b_1, b_2, \ldots, b_{k_n-1})$ such that the number of $b_i$'s not in $H$ is exactly $j$. 
By Lemma \ref{lem: lem for 4.1}, we have
\begin{align*}
n(k_n-1)_a &\le \sum_{j=2}^{k_n-1} k_n^j(n + |H|t_n)^{k_n-1-j}r_n^j|G|^{j-1} \\
& \le k_n^2(n + |H|t_n)^{k_n-3}r_n^2|G|\frac{1}{1-\frac{k_nr_n|G|} {n+|H|t_n}} \\
& = k_n^2\left(n^{k_n-3} + O(n^{k_n-4}t_nk_n)\right)r_n^2|G|\frac{1}{1-\frac{k_nr_n|G|} {n+|H|t_n}} \\
& = O(n^{k_n-3}r_n^2k_n^2).
\end{align*}
Therefore, it follows that
$$
\mu_{\unl{n}}(a) = \frac{n(k_n-1)_a}{n^{k_n-1}} \le O\left(\frac{k_n^2r_n^2}{n^2}\right)
$$
so when $n$ is sufficiently large, we have 
$$
\mu_{\unl{n}}(a) \le n^{-\frac{5}{3}}.
$$
Also, note that $1/n\le \nu_{\unl{n}}(a) \le r_n/n$. 
Let $p:= \nu_{\unl{n}}(a)$ and $q:=\mu_{\unl{n}}(a)$. By \cite[Lemma 2.2]{Mes23}, we have
$$
D_{\KL}(\nu_{\unl{n}}||\mu_{\unl{n}}) \ge p \log \frac{p}{q} + (1-p) \log \frac{1-p}{1-q}. 
$$
Regarding the right hand side as a function on $q$, we see that it is decreasing on $(0,p]$. If $n$ is sufficiently large, we have $q \le n^{-\frac{5}{3}} \le n^{-1} \le p$. Hence, for a large enough $n$,
$$
D_{\KL}(\nu_{\unl{n}}||\mu_{\unl{n}}) \ge p \log pn^{\frac{5}{3}} + (1-p)\log\frac{1-p}{1-n^{-\frac{5}{3}}} \ge p \log n^{\frac{2}{3}} + (1-p) \log(1-p). 
$$
For sufficiently large $n$, we have $p \in [0, \frac{1}{2}]$ so $\log(1-p) \ge (-2\log 2)p$. It follows that for large enough $n$,
\begin{equation*}
D_{\KL}(\nu_{\unl{n}}||\mu_{\unl{n}}) \ge p\left(\frac{2}{3}\log n - 2\log2\right) \ge \frac{\log n}{2n}. \qedhere
\end{equation*}

\end{proof}

Recall that
$$
E(\unl{n}) = \al(\unl{n})\frac{\det(M)}{k_n\prod_{a \in G_+}n(k_n-1)_a}\exp(-n D_{\KL}(\nu_{\unl{n}}||\mu_{\unl{n}})). 
$$

\begin{lem}
\label{lem: lemma 5.4 of Mes23}
Suppose that for every $\ep >0$, $k_n < n^{\ep}$ for all sufficiently large $n$. Then for every $\xi >0$, the following holds for all $\unl{n} \in B_2(n,H)$.
$$
\frac{\det(M)}{\prod_{a \in G_+} n(k_n-1)_a} \exp(-nD_{\KL}(\nu_{\unl{n}}||\mu_{\unl{n}})) \le O\left(\frac{1}{n^{\frac{1}{2}-\xi }}\right).
$$
\end{lem}

\begin{proof}
Suppose that $\ell=0$ ($\ell$ is defined as above). Since $G_+$ generates $G$, we can choose $a \in G_+ \backslash H$. For this $a$, we have $(-a+H) \cap G_+ = \emptyset$ as $\ell=0$. By Lemma \ref{lem: lemma 5.3 of Mes23} and \eqref{eq: 3.6.1}, we have 
$$
\frac{\det(M)}{\prod_{a \in G_+} n(k-1)_a} \exp(-nD_{\KL}(\nu_{\unl{n}}||\mu_{\unl{n}})) 
\le \frac{k_n^{|G|}}{\sqrt{n}} = O\left(\frac{1}{n^{\frac{1}{2}-\xi}}\right).
$$
If $\ell>0$, then Gibbs' inequality (\cite[Lemma 2.1]{Mes23}) together with Lemma \ref{lem: lemma 5.2 of Mes23} implies that 
\begin{equation*}
\frac{\det(M)}{\prod_{a \in G_+} n(k_n-1)_a} \exp(-nD_{\KL}(\nu_{\unl{n}}||\mu_{\unl{n}})) 
\le \frac{\det(M)}{\prod_{a \in G_+} n(k_n-1)_a}
\le O\left(k_n^{|G|}r_n^{|G|-|H|} \frac{t_n}{n}\right) = O\left(\frac{1}{n^{\frac{1}{2}-\xi}}\right). \qedhere
\end{equation*}
\end{proof}

\begin{rem}
In Lemma \ref{lem: lemma 5.4 of Mes23}, we should assume that $k_n \ll n^\ep$
for arbitrary $\ep >0$ to obtain a sufficiently strong bound on $k_n^{|G|}$ (note that $|G|$ can be arbitrarily large). 
This is exactly why the same upper bound assumption on $k_n$ was required in the statements of our main theorems as well. 
In the other parts of the paper, it suffices to assume the weaker bound
$k_n= O(n^{\frac{1}{30} - \delta})$ for some $\delta>0$.
\end{rem}

\begin{prop}
\label{prop: moment for B2}    
Suppose that for every $\ep >0$, $k_n < n^{\ep}$ for all sufficiently large $n$. Then
$$
\lim_{n\to \infty} \sum_{\unl{n} \in B_2(n,H)} E(\unl{n}) = 0. 
$$
\end{prop}

\begin{proof}
By Lemma \ref{lem: size of ball for H} and the assumption on $k_n$, it follows that
$$
|B_2(n,H)| \le |B(n,H)| = O\left(n^{\frac{|H|}{2} - \frac{1}{2}+\xi}\right).
$$
for every $\xi>0$. 
By \eqref{eq: Stirling formula for alpha} and Lemma \ref{lem: lemma 5.4 of Mes23}, we have
$$
E(\unl{n}) = O\left(n^{\xi - \frac{|H|}{2}}\right)
$$
for every $\xi>0$. Now we have
$$
\sum_{\unl{n} \in B_2(n,H)} E(\unl{n}) = O\left(n^{2\xi-\frac{1}{2}}\right)
$$
and we complete the proof by taking $\xi < \frac{1}{4}$. 
\end{proof}

\section{Convergence of moments and convergence to the Cohen--Lenstra distribution} \label{Sec_CL all primes}
In this section, we prove our main theorems. We first prove that the moments of $\cok(A_n)$ converge to $1$ under certain assumptions. As remarked in Section \ref{Sec1}, this implies the convergence of $\cok(A_n)$ to the Cohen--Lenstra distribution by Wood's theorem \cite[Theorem 3.1]{Woo19}.

\begin{thm}
\label{thm: moment theorem 1}
Let $G$ be a finite abelian group. Assume that a sequence $(k_n)_{n=1}^{\infty}$ satisfies the following:
\begin{enumerate}
    \item $\gcd(|G|, k_n)=1$ for all sufficiently large $n$;
    \item for every $\ep>0$, $k_n < n^{\ep}$ for all sufficiently large $n$;
    \item if $|G|$ is even, then for every $\dt>0$, $\dt \log\log n < k_n$ for all sufficiently large $n$.
\end{enumerate}
Then 
\begin{equation}
\label{eq: moment theorem}
\lim_{n \to \infty} \bE(\#\Sur(\cok(A_n), G)) = 1.    
\end{equation}
\end{thm}

\begin{proof}
 By assumption (1) and Proposition \ref{prop: moments sum splits}, it is enough to show that
$$
\lim_{n \to \infty}\sum_{H \in \text{Sub}(G)} \sum_{\unl{n} \in B(n, H)} E(\unl{n}) = 1. 
$$
By assumption (2) and Proposition \ref{prop: main moment}, we have
$$
\lim_{n \to \infty} \sum_{\unl{n} \in B(n, G)} E(\unl{n}) = 1. 
$$
Let $H$ be a proper subgroup of $G$ and $B_1(n,H)$ and $B_2(n,H)$ be as in the beginning of Section \ref{Sec_B1(n,H)}. By assumption (2) and Proposition \ref{prop: moment for B2}, we have
$$
\lim_{n \to \infty} \sum_{\unl{n} \in B_2(n, H)} E(\unl{n}) = 0. 
$$
If $|G|$ is odd, then $B_2(n,H) = B(n,H)$ and this finishes the proof. Suppose that $|G|$ is even. By assumption (2), (3) and Proposition \ref{prop: moment for B1}, we have
$$
\lim_{n \to \infty} \sum_{\unl{n} \in B_1(n, H)} E(\unl{n}) = 0.
$$
This completes the proof. 
\end{proof}

\begin{thm}
\label{thm: cokernel distribution theorem 1}
Let $G$ be a finite abelian group and $\cP$ be a finite set of primes including those dividing $|G|$. 
Assume that a sequence $(k_n)_{n=1}^{\infty}$ satisfies the following:
\begin{enumerate}
    \item for every prime $p$ in $\cP$, $p \nmid k_n$ for all sufficiently large $n$;
    \item for every $\ep>0$, $k_n < n^{\ep}$ for all sufficiently large $n$;
    \item if $2 \in \cP$, then for every $\dt>0$, $\dt \log\log n < k_n$ for all sufficiently large $n$.
\end{enumerate}
Then
$$
\lim_{n\to \infty} \bP\left(\bigoplus_{p \in \cP} \cok(A_n)_p \cong G\right) = \frac{1}{|\Aut(G)|}\prod_{p\in \cP} \prod_{i=1}^\infty (1-p^{-i})
= \prod_{p \in \cP} \nu_{\CL, p}(G_p).
$$
\end{thm}

The following two corollaries are special cases of Theorem \ref{thm: moment theorem 1} and \ref{thm: cokernel distribution theorem 1}, respectively. 

\begin{cor}
Suppose that the following hold: 
\begin{enumerate}
\item 
for every prime $p$, $p \nmid k_n$ for all sufficiently large $n$;
\item 
for every $\ep>0$, $k_n < n^{\ep}$ for all sufficiently large $n$;
\item 
for every $\dt >0$, $\dt \log\log n < k_n$ for all sufficiently large $n$. 
\end{enumerate}
Then for every finite abelian group $G$, we have 
$$
\lim_{n \to \infty} \bE(\#\Sur(\cok(A_n), G)) = 1.
$$
\end{cor}

\begin{cor}
Suppose that the following hold: 
\begin{enumerate}
\item 
for every prime $p$, $p \nmid k_n$ for all sufficiently large $n$;
\item 
for every $\ep>0$, $k_n < n^{\ep}$ for all sufficiently large $n$;
\item 
for every $\dt >0$, $\dt \log\log n < k_n$ for all sufficiently large $n$. 
\end{enumerate}
Let $S$ be a finite set of primes and for each $p \in S$, let $G_p$ be a finite abelian $p$-group. Then
$$
\lim_{n \to \infty} \bP\left(\bigoplus_{p\in S}\cok(A_n)_p \cong \bigoplus_{p\in S} G_p\right) = \prod_{p\in S} \frac{1}{|\Aut(G_p)|}\prod_{j=1}^{\infty}(1-p^{-j})=  \prod_{p \in S} \nu_{\CL, p}(G_p).
$$
\end{cor}

\begin{rem}
Does the conclusion of Theorem \ref{thm: moment theorem 1} still hold with assumption (3) replaced by a weaker condition $\lim_{n \to \infty} k_n = \infty$? In Proposition \ref{prop: partial answer}, we prove that \eqref{eq: moment theorem} holds in the special case $G = \Z/2\Z$ when we only assume  $\lim_{n \to \infty} k_n = \infty$ instead of assumption (3). 
We will also show below why at least the condition that $\lim_{n \to \infty} k_n = \infty$ is necessary. See Proposition \ref{prop: kn has to go to infinity} for this. 
\end{rem}

\begin{prop}
\label{prop: partial answer}    
Suppose that the following hold:  
\begin{enumerate}
\item 
$2 \nmid k_n$ for all sufficiently large $n$;
\item 
$k_n = O(n^{\frac{1}{30}-\delta})$ for some $\delta>0$;
\item 
$\lim_{n \to \infty} k_n = \infty$.
\end{enumerate}
Then
\begin{equation*}
\lim_{n \to \infty} \bE(\#\Sur(\cok(A_n), \Z/2\Z)) = 1.     
\end{equation*}
\end{prop}

\begin{proof}
Let $H = \{0\} \le \Z/2\Z$. By Proposition \ref{prop: moments sum splits} and \ref{prop: main moment}, it is enough to show that
$$
\lim_{n \to \infty} \sum_{\unl{n} \in B(n,H)} E(\unl{n}) = 0.
$$
Let $\unl{n} \in B(n,H)$ with $n_0 = n- \ell$ and $n_1 = \ell$. By definition we have $1 \le \ell \le r_n = (k_n-1)^2C_n|G|\log n$. 
By \eqref{eq: 3.6.1}, 
$$
\det(M) \le k_n^2 n(k_n-1)_0 n(k_n-1)_1. 
$$
So it follows from \eqref{eq: E formula 1} that
$$
E(\unl{n}) \le \frac{n^\ell}{\ell!}\frac{k_n^2 n(k_n-1)_0^{n - \ell}n(k_n-1)_1^{\ell}}{k_nn^{(k_n-1)n}} =: U_n(\ell). 
$$
Note that
\begin{align*}
n(k_n-1)_0 = \sum_{\substack{0\le i \le k_n-1  \\i ~\text{even}}} \binom{k_n-1}{i}(n-\ell)^{k_n-1-i}\ell^{i} &= \frac
{(n-\ell+\ell)^{k_n-1} + (n-\ell-\ell)^{k_n-1}}{2}  \\
 &= \frac{n^{k_n-1}+ (n-2\ell)^{k_n-1}}{2}.    
\end{align*}
Similarly, we have
$$
n(k_n-1)_1 = \frac{n^{k_n-1} - (n-2\ell)^{k_n-1}}{2}. 
$$
By Lemma \ref{lem: lemma for partial answer}, there exists $c>0$ such that for all $n > c$ (assume additionally that $c$ is large enough so that $n>c$ implies that $\frac{1}{3} \le \frac{n-r_n}{2n}$)
$$
U_n(\ell) \le \frac{k_n}{\ell!}\left(1 - \frac{\ell(k_n-1)}{2n} \right)^{n-\ell}(\ell(k_n-1))^{\ell} \le \frac{\ell^\ell}{\ell!}\frac{k_n^{\ell+1}}{e^{\frac{(n-\ell)\ell(k_n-1)}{2n}}} \le \frac{D e^\ell k_n^{\ell+1}}{e^{\frac{\ell(k_n-1)}{3}}}
$$
for a fixed constant $D>0$ (by Stirling's formula). When $k_n \ge 12$, we have $k_n e^{-\frac{k_n-4}{3}} < 1$ so
$$
\sum_{\unl{n} \in B(n,H)} E(\unl{n}) 
\le \sum_{1 \le \ell \le r_n} U_n(\ell) 
\le D\sum_{\ell = 1}^{\infty} k_n\left(k_n e^{-\frac{k_n-4}{3}} \right)^\ell 
= \frac{Dk_n^2 e^{-\frac{k_n-4}{3}}}{1- k_n e^{-\frac{k_n-4}{3}}}.
$$
The right-hand side tends to $0$ as $k_n \to \infty$.
\end{proof}

\begin{lem}
\label{lem: lemma for partial answer}
Suppose that $k_n = O(n^{\frac{1}{7}-\delta})$ for some $\delta>0$. Then there exists $c>0$ such that the following holds for all $n> c$ and $1\le \ell \le r_n$. 
\begin{enumerate}
\item 
    $$
    \frac{n(k_n-1)_0}{n^{k_n-1}} =  \frac{n^{k_n-1}+ (n-2\ell)^{k_n-1}}{2n^{k_n-1}} \le 1 - \frac{\ell(k_n-1)}{2n} 
    $$
\item 
 $$
    \frac{n(k_n-1)_1}{n^{k_n-2}} =   \frac{n^{k_n-1} - (n-2\ell)^{k_n-1}}{2n^{k_n-2}} \le \ell(k_n-1)
    $$   
\end{enumerate}
\end{lem}

\begin{proof}
By binomial theorem, we have
$$
(n-2\ell)^{k_n-1} = n^{k_n-1} - 2\ell(k_n-1)n^{k_n-2} + \sum_{i=2}^{k_n-1}\binom{k_n-1}{i}(-2\ell)^{i}n^{k_n-1-i}. 
$$
We see that
$$
\sum_{i=2}^{k_n-1}\binom{k_n-1}{i}(-2\ell)^{i}n^{k_n-1-i}\le \sum_{i=2}^{k_n-1}(k_n-1)^i(2\ell)^{i}n^{k_n-1-i} \le \frac{n^{k_n-3}(k_n-1)^2(2\ell)^2}{1-\frac{(k_n-1)2\ell}{n}} \le \ell(k_n-1)n^{k_n-2}
$$
for all sufficiently large $n$. Note that the assumption $k_n = O(n^{\frac{1}{7}-\delta})$ is used to justify the last two inequalities. Now (1) follows. For (2), note similarly as above that
$$
\left|\sum_{i=3}^{k_n-1}\binom{k_n-1}{i}(-2\ell)^{i}n^{k_n-1-i}\right| \le \sum_{i=3}^{k_n-1}(k_n-1)^i(2\ell)^{i}n^{k_n-1-i} \le \frac{n^{k_n-4}(k_n-1)^3(2\ell)^3}{1-\frac{(k_n-1)2\ell}{n}} \le \binom{k_n-1}{2}(-2\ell)^{2}n^{k_n-3}
$$
for all sufficiently large $n$. Now it is straightforward to see that (2) holds.
\end{proof}

\begin{prop}
\label{prop: kn has to go to infinity}
Suppose that the following hold:  
\begin{enumerate}
\item 
$2 \nmid k_n$ for all sufficiently large $n$;
\item 
$k_n = O(n^{\frac{1}{30}-\delta})$ for some $\delta>0$.
\end{enumerate}
If $\lim_{n \to \infty} k_n \neq \infty$, then there exists $\eta>0$ such that the following holds for infinitely many $n$:
$$
\bE(\#\Sur(\cok(A_n), \Z/2\Z)) > 1+ \eta.
$$    
\end{prop}

\begin{proof}
Let $G = \Z/2\Z$ and $H= \{0\} \le G$. Let $d>3$ be a constant and suppose that $k_n < d$ for infinitely many $n$. By Proposition \ref{prop: main moment}, we have
$$
\lim_{n \to \infty} \sum_{\unl{n} \in B(n,G)} E(\unl{n}) =1. 
$$
By Proposition \ref{prop: moments sum splits}, it is enough to show that there exists $\eta>0$ such that
$$
\sum_{\unl{n} \in B(n,H)} E(\unl{n}) > \eta
$$
for infinitely many $n$. Let $\unl{n} \in B(n,H)$ be such that $n_0 = n-1$ and $n_1 = 1$. Then by \eqref{eq: E formula 1}, 
$$
E(\unl{n}) = \frac{n}{n^{(k_n-1)n}k_n}\det(M) n(k_n-1)_0^{n-2},
$$
where
$$
M = \begin{pmatrix}
    (k_n-1)(n-1)n(k_n-2)_0+n(k_n-1)_0 & (k_n-1)\sqrt{n-1}n(k_n-2)_1\\
    (k_n-1)\sqrt{n-1}n(k_n-2)_1 & (k_n-1)n(k_n-2)_0 + n(k_n-1)_1
  \end{pmatrix}.
$$
As in the proof of Proposition \ref{prop: partial answer}, for $a \in \{1,2\}$ we have
$$
n(k_n-a)_0 = \frac{n^{k_n-a}+ (n-2)^{k_n-a}}{2}
$$
and
$$
n(k_n-a)_1 = \frac{n^{k_n-a} - (n-2)^{k_n-a}}{2}.
$$
Then it is straightforward to see that the following hold: 
\begin{enumerate}
\item
$n(k_n-1)_0 \ge (n-1)n(k_n-2)_1$ for all $n\ge 1$;
\item 
$n(k_n-2)_0 \ge (k_n-1)n(k_n-2)_1$ for all $n$ that are sufficiently large and satisfy $k_n<d$. 
\end{enumerate}
This yields that for $n$ as in (2), 
\begin{align*}
E(\unl{n}) & \ge \frac{n(n-1)(k_n-1)^2n(k_n-2)_0^2 n(k_n-1)_0^{n-2}}{n^{(k_n-1)n}k_n}  \\ 
& \ge \frac{n(n-1)(k_n-1)^2 (n-1)^{2(k_n-2)} (n-1)^{(n-2)(k_n-1)}}{n^{(k_n-1)n}k_n} \\
& \ge \frac{(k_n-1)^2 }{k_n} \left(\frac{n-1}{n}\right)^{n(k_n-1)} \\
& \ge \frac{(k_n-1)^2 }{4^{k_n-1}k_n}.
\end{align*}
It is clear that the right-hand side is bounded below by some constant $\eta >0$ when $3\le k_n < d$. This completes the proof. 
\end{proof}

\section{The limiting distribution of \texorpdfstring{$\cok(A_n)_2$}{cok(An)} is not Cohen--Lenstra when \texorpdfstring{$k_n$}{kn} is a constant} \label{Sec_not CL}

Let $\ol{A_n} \in \M_n(\bF_2)$ be the reduction of $A_n$ modulo $2$. If the $2$-Sylow subgroup of $\cok(A_n)$, denoted by $\cok(A_n)_2$, converges to the Cohen--Lenstra distribution, then \cite[Theorem 6.3]{CL84} implies that 
\begin{equation} \label{eq7a}
\underset{n \to \infty}{\lim} \bP(\dim_{\bF_2} \ker \ol{A_n} = r) 
=  2^{-r^2} \prod_{k=1}^{r} (1-2^{-i})^{-2} \prod_{i=1}^{\infty} (1-2^{-i}) = O(2^{-r^2})
\end{equation}
for every $r \ge 0$. Comparing \eqref{eq7a} with the following theorem, we deduce that $\cok(A_n)_2$ does not converge to the Cohen--Lenstra distribution as $n \to \infty$ when $k_n = k$ for a fixed constant $k \ge 3$. (If $k_n=k$ for an even integer $k \ge 4$, then $\bP(\cok(\ol{A_n})=0)=0$ as the row space $\mathrm{row}(\ol{A_n})$ is contained in a proper subspace $\{(v_1, \ldots, v_n) \in \bF_2^n: \sum_{i=1}^n  v_i = 0\}$ of $\bF_2^n$. As a result, $\cok(A_n)_2$ does not converge to the Cohen--Lenstra distribution.
Thus, it suffices to consider the case where $k$ is odd as in the following theorem.) 

\begin{thm} \label{thm7a}
Let $k \ge 3$ be an odd integer, $k_n=k$ for all $n$ and $r$ be a positive integer. Then for all sufficiently large $n$, 
\begin{equation} \label{eq7b}
\bP(\dim_{\bF_2} \ker \ol{A_n} \ge r) 
\ge \frac{1}{4r!} \left ( \frac{2(k-1)}{e^{k-1}} \right )^r.
\end{equation}
\end{thm}

Let $T_n := \{ K \subset [n]^k : |K|=n \}$ and
$$
p(S) := \sum_{K \in S} \bP(X_n = K) = \sum_{K \in S} \frac{\det(B_n[K])^2}{\det(B_n^T B_n)}
$$
for each $S \subset T_n$. For each $i \in [n]$, let $T_{n,i}$ be the set of $K \in T_n$ such that the $i$-th column of $B_n[K]$ is $2e_j^T$ for some $j \in [n]$. Equivalently, 
$$
T_{n,i} = \left \{ K \in T_n : \; \begin{matrix}
\text{there exists } x=(x_1, \ldots, x_k) \in K \text{ such that } |\{ t : x_t=i \}|=2  \\
\text{and } y_1, \ldots, y_k \neq i \text{ for every } y=(y_1, \ldots, y_k) \in K \backslash \{ x \}
\end{matrix}\right \}.
$$
Before proving Theorem \ref{thm7a}, we provide several lemmas. 

\begin{lem} \label{lem7c}
$$
\bP \left ( \dim_{\bF_2} \ker \ol{A_n} \ge r \right ) \ge p \left ( \bigcup_{i_1 < \cdots < i_r} (T_{n, i_1} \cap \cdots \cap T_{n, i_r}) \right ).
$$
\end{lem}

\begin{proof}
If $K \in T_{n, i_1} \cap \cdots \cap T_{n, i_r}$, then each of $i_1, \ldots, i_r$-th columns of $\ol{B_n[K]}$ is zero so $\dim_{\bF_2} \ker \ol{B_n[K]} \ge r$. 
\end{proof}

\begin{lem} \label{lem7d}
For every $1 \le r \le n-1$, we have
\begin{equation*}
p \left ( \bigcup_{i_1 < \cdots < i_r} (T_{n, i_1} \cap \cdots \cap T_{n, i_r}) \right ) 
\ge \sum_{i_1 < \cdots < i_r} p(T_{n, i_1} \cap \cdots \cap T_{n, i_r}) - r \sum_{i_1 < \cdots < i_{r+1}} p(T_{n, i_1} \cap \cdots \cap T_{n, i_{r+1}})
\end{equation*}
\end{lem}

\begin{proof}
For every subset $I \subset [n]$, let $T_{n, I}$ be the set of the elements $K \in T_n$ such that $K \in T_{n,i}$ for every $i \in I$ and $K \notin T_{n,i}$ for every $i \in [n] \backslash I$. Then 
\begin{align*}
& \sum_{i_1 < \cdots < i_r} p(T_{n, i_1} \cap \cdots \cap T_{n, i_r}) - r \sum_{i_1 < \cdots < i_{r+1}} p(T_{n, i_1} \cap \cdots \cap T_{n, i_{r+1}}) \\
= & \, \sum_{m=r}^{n} \sum_{\substack{I \subset [n] \\ |I|=m}} \binom{m}{r} p(T_{n, I}) - r \sum_{m=r+1}^{n} \sum_{\substack{I \subset [n] \\ |I|=m}} \binom{m}{r+1} p(T_{n, I}) \\
= & \, \sum_{\substack{I \subset [n] \\ |I|=r}} p(T_{n, I}) + \sum_{m=r+1}^{n} \sum_{\substack{I \subset [n] \\ |I|=m}} \left ( \binom{m}{r}-r\binom{m}{r+1} \right ) p(T_{n, I}) \\
\le & \, \sum_{m=r}^{n} \sum_{\substack{I \subset [n] \\ |I|=m}} p(T_{n, I}) \\
= & \, p \left ( \bigcup_{i_1 < \cdots < i_r} (T_{n, i_1} \cap \cdots \cap T_{n, i_r}) \right ). \qedhere
\end{align*}
\end{proof}

The next lemma is the key part of the proof of Theorem \ref{thm7a}. Denote (see \eqref{eq: det of BnTBn})
$$
C_{n,k} := \det(B_n^T B_n) = k^{n+1}n^{(k-1)n}.
$$
\begin{lem} \label{lem7e}
For every $1 \le i_1 < \cdots < i_r \le n$, we have
$$
p(T_{n, i_1} \cap \cdots \cap T_{n, i_r}) 
= (2k(k-1)(n-r)^{k-2})^r \frac{C_{n-r, k}}{C_{n,k}}.
$$
\end{lem}

\begin{proof}
Without loss of generality, we may assume that $(i_1, \ldots, i_r) = (1, \ldots, r)$. Let $\prec$ be any ordering on $[n]^k$ such that $(x_1, \ldots, x_k) \prec (y_1, \ldots, y_k)$ if 
$$
\min(x_1, \ldots, x_k) < \min(y_1, \ldots, y_k).
$$
Assume that the rows of $B_n$ are ordered by the ordering $\prec$. If $K \in \bigcap_{i=1}^{r} T_{n,i}$ and $\det(B_n[K]) \neq 0$, then the $i$-th column of $B_n[K]$ is given by $2e_i^T$ for each $i \in [r]$ by the choice of the ordering of the rows of $B_n$. Precisely, we have
$$
B_n[K] = \begin{pmatrix}
2I_r & * \\
O & B_{n-r}[K_2] \\
\end{pmatrix} \in \M_{r+(n-r)}(\Z)
$$
for some $K_2 \in T_{n-r}$ such that $\det(B_{n-r}[K_2]) \neq 0$. This implies that
\begin{align*}
p \left ( \bigcap_{i=1}^{r} T_{n,i} \right ) 
& = \sum_{K \in \bigcap_{i=1}^{r} T_{n,i}} \frac{\det(B_n[K])^2}{\det(B_n^T B_n)} \\
& = \sum_{\substack{K \in \bigcap_{i=1}^{r} T_{n,i} \\ \det(B_n[K]) \neq 0} }\frac{\det(B_n[K])^2}{\det(B_n^T B_n)} \\
& = \left | U_{n,k,r} \right | \sum_{\substack{K_2 \in T_{n-r} \\ \det(B_{n-r}[K_2]) \neq 0}}\frac{(2^r \det(B_{n-r}[K_2]))^2}{\det(B_n^T B_n)} \\
& = \left | U_{n,k,r} \right | \sum_{K_2 \in T_{n-r}}\frac{(2^r \det(B_{n-r}[K_2]))^2}{\det(B_n^T B_n)} \\
& = \left | U_{n,k,r} \right | \frac{4^r C_{n-r, k}}{C_{n,k}},
\end{align*}
where 
$$
U_{n,k,r} := \left \{ K_1 \subset [n]^k : |K_1|=r \text{ and } B_n[K_1] = \begin{pmatrix}
2I_r & * \\
\end{pmatrix} \right \}.
$$
Let $K_1 = \{\mathbf{x}_1, \ldots, \mathbf{x}_r\} \in U_{n,k,r}$ ($\mathbf{x}_1 \prec \mathbf{x}_2 \prec \cdots \prec \mathbf{x}_r$) and $\mathbf{x}_i = (x_{i,1}, \ldots, x_{i,k})$. Then for each $i \in [r]$, exactly two of $x_{i,1}, \ldots, x_{i,k}$ are equal to $i$ (there are $\binom{k}{2}$ choices) and the other $x_{i, j}$'s are larger than $r$ (there are $(n-r)^{k-2}$ choices). Now we have
$$
|U_{n,k,r}| = \left ( \binom{k}{2} (n-r)^{k-2} \right )^r
$$
so
\begin{equation*}
p \left ( \bigcap_{i=1}^{r} T_{n,i} \right ) = \left ( \binom{k}{2} (n-r)^{k-2} \right )^r \frac{4^r C_{n-r, k}}{C_{n,k}} = (2k(k-1)(n-r)^{k-2})^r \frac{C_{n-r, k}}{C_{n,k}}. \qedhere
\end{equation*}
\end{proof}

\bigskip
\begin{proof}[Proof of Theorem \ref{thm7a}]
By Lemma \ref{lem7c}, \ref{lem7d} and \ref{lem7e}, we have
\begin{align*}
& \, \bP \left ( \dim_{\bF_2} \ker \ol{A_n} \ge r \right ) \\
\ge & \, \sum_{i_1 < \cdots < i_r} p(T_{n, i_1} \cap \cdots \cap T_{n, i_r}) - r \sum_{i_1 < \cdots < i_{r+1}} p(T_{n, i_1} \cap \cdots \cap T_{n, i_{r+1}}) \\
= & \,  \binom{n}{r}(2k(k-1)(n-r)^{k-2})^r \frac{C_{n-r, k}}{C_{n,k}} - r\binom{n}{r+1}(2k(k-1)(n-r-1)^{k-2})^{r+1} \frac{C_{n-r-1, k}}{C_{n,k}} \\
= & \,  \binom{n}{r}(2k(k-1)(n-r)^{k-2})^r \frac{C_{n-r, k}}{C_{n,k}} \left ( 1 - \frac{2k(k-1)r(n-r)}{r+1} \frac{(n-r-1)^{(k-2)(r+1)}}{(n-r)^{(k-2)r}} \frac{C_{n-r-1,k}}{C_{n-r,k}} \right ).
\end{align*}
By the formula $C_{n,k}=k^{n+1}n^{(k-1)n}$, we have
$$
\frac{C_{n-r-1,k}}{C_{n-r,k}} 
= \left ( 1 - \frac{1}{n-r} \right )^{(n-r)(k-1)} \frac{1}{k(n-r-1)^{k-1}} 
\sim \frac{1}{ke^{k-1}(n-r-1)^{k-1}}
$$
and
\begin{align*}
\frac{2k(k-1)r(n-r)}{r+1} \frac{(n-r-1)^{(k-2)(r+1)}}{(n-r)^{(k-2)r}} \frac{C_{n-r-1,k}}{C_{n-r,k}} 
\sim \frac{2(k-1)}{e^{k-1}} \frac{r}{r+1} < \frac{2}{3}.
\end{align*}
This implies that if $n$ is sufficiently large (in terms of $r$ and $k$), we have
$$
\bP \left ( \dim_{\bF_2} \ker \ol{A_n} \ge r \right ) > \frac{1}{3} \binom{n}{r}(2k(k-1)(n-r)^{k-2})^r \frac{C_{n-r, k}}{C_{n,k}}. 
$$
We also have
$$
\frac{C_{n-r, k}}{C_{n,k}}
= \left ( 1 - \frac{r}{n} \right )^{n(k-1)} \frac{1}{k^r (n-r)^{(k-1)r}}
\sim \frac{1}{k^r e^{(k-1)r}(n-r)^{(k-1)r}}
$$
so
$$
\binom{n}{r}(2k(k-1)(n-r)^{k-2})^r \frac{C_{n-r, k}}{C_{n,k}}
\sim \frac{1}{r!}\left ( \frac{2(k-1)}{e^{k-1}} \right )^r.
$$
We conclude that for all sufficiently large $n$, 
\begin{equation*}
\bP \left ( \dim_{\bF_2} \ker \ol{A_n} \ge r \right ) \ge \frac{1}{4r!} \left ( \frac{2(k-1)}{e^{k-1}} \right )^r. \qedhere 
\end{equation*}
\end{proof}

\section*{Acknowledgments}
\hspace{3mm} 
We are grateful to András Mészáros for helpful comments. 
Jungin Lee was supported by the National Research Foundation of Korea (NRF) grant funded by the Korea government (MSIT) (No. RS-2024-00334558 and No. RS-2025-02262988) and by the Ajou University Research Fund (S-2025-G0001-00398).
Myungjun Yu was supported by the National Research Foundation of Korea (NRF)
grant funded by the Korea government (MSIT) (RS-2025-23525445), by Yonsei University Research Fund (2024-22-0146), and by Korea Institute for Advanced Study (KIAS) grant funded by the Korea government.
\


\begin{thebibliography}{999999}

\bibitem{CL84}
H. Cohen and H. W. Lenstra Jr., Heuristics on class groups of number fields, Number Theory, Noordwijkerhout 1983, Lecture Notes in Math. 1068, Springer, Berlin, 1984, 33--62.

\bibitem{CS04}
I. Csiszár and P. C. Shields, Information theory and statistics: A tutorial, Found. Trends Commun. Inform. Theory 1 (2004), no. 4, 417--528.

\bibitem{HJ13}
R. A. Horn and C. R. Johnson, Matrix Analysis, 2nd ed. Cambridge University Press, 2013.

\bibitem{KLNP20}
M. Kahle, F. H. Lutz, A. Newman and K. Parsons, Cohen--Lenstra heuristics for torsion in homology of random complexes, Exp. Math. 29 (2020), no. 3, 347--359.

\bibitem{KN22}
M. Kahle and A. Newman, Topology and geometry of random $2$-dimensional hypertrees, Discrete Comput. Geom. 67 (2022), no. 4, 1229--1244.

\bibitem{Kal83}
G. Kalai, Enumeration of $\bQ$-acyclic simplicial complexes, Israel J. Math. 45 (1983), no. 4, 337--351.

\bibitem{KLY24}
D. Y. Kang, J. Lee and M. Yu, Random $p$-adic matrices with fixed zero entries and the Cohen--Lenstra heuristics, arXiv:2409.01226.

\bibitem{Mes23}
A. Mészáros, Cohen--Lenstra distribution for sparse matrices with determinantal biasing, Int. Math. Res. Not. (2025), no. 3, rnae292.

\bibitem{Mes24a}
A. Mészáros, The $2$-torsion of determinantal hypertrees is not Cohen--Lenstra, arXiv:2404.02308, to appear in Israel J. Math.

\bibitem{Mes24b}
A. Mészáros, A phase transition for the cokernels of random band matrices over the $p$-adic integers, arXiv:2408.13037.

\bibitem{Woo17}
M. M. Wood, The distribution of sandpile groups of random graphs, J. Amer. Math. Soc. 30 (2017), no. 4, 915--958.

\bibitem{Woo19}
M. M. Wood, Random integral matrices and the Cohen--Lenstra heuristics, Amer. J. Math. 141 (2019), no. 2, 383--398.
\end{thebibliography}
\end{document}